\documentclass[a4paper,12pt,final]{article}

\usepackage{amsmath,amsthm,amssymb}
\usepackage[a4paper,margin=1in]{geometry} 

\usepackage[utf8]{inputenc}
\usepackage[T1]{fontenc}
\usepackage{lmodern}

\usepackage{dsfont}

\usepackage{mathrsfs}

\usepackage{stmaryrd}

\usepackage{verbatim}

\usepackage[lowtilde]{url}

\usepackage{paralist}

\usepackage{xcolor}

\newtheorem{theorem}{Theorem}[section]
\newtheorem{proposition}[theorem]{Proposition}
\newtheorem{corollary}[theorem]{Corollary}
\newtheorem{lemma}[theorem]{Lemma}

\theoremstyle{definition}
\newtheorem{definition}[theorem]{Definition}
\newtheorem{assumption}[theorem]{Assumption}
\newtheorem{example}[theorem]{Example}

\theoremstyle{definition}
\newtheorem{remark}[theorem]{Remark}
\newtheorem{notation}[theorem]{Notation}

\swapnumbers
\numberwithin{equation}{section}

\newcommand\bC{\mathds{C}}
\newcommand\bB{\mathds{B}}

\newcommand\bE{\mathds{E}}

\newcommand\bN{\mathds{N}}
\newcommand\bP{\mathds{P}}
\newcommand\bR{\mathds{R}}

\newcommand{\one}{\mathds 1}

\newcommand\cB{\mathcal{B}}
\newcommand\cC{\mathcal{C}}
\newcommand\cD{\mathcal{D}}

\newcommand\cF{\mathcal{F}}

\newcommand\cP{\mathcal{P}}

\newcommand\cO{\mathcal{O}}
				
\newcommand{\nnrm}[2]{\ensuremath{\| #1 \|_{#2}}}

\newcommand{\gnnrm}[2]{\ensuremath{\big\| #1 \big\|_{#2}}}

\renewcommand{\epsilon}{\ensuremath{\varepsilon}}

\renewcommand{\geq}{\ensuremath{\geqslant}}
\renewcommand{\leq}{\ensuremath{\leqslant}}

\newcommand{\dl}{\ensuremath{\mathrm{d}}}

\newcommand{\supp}{\ensuremath{\mathop{\operatorname{supp}}}}

\newcommand{\bo}{\ensuremath{\mathscr L}}

\newcommand{\tr}{\ensuremath{\mathop{\operatorname{Tr}}}}

\allowdisplaybreaks

\title
{
Weak error analysis via functional It\^o calculus
}
\author{Mih\'aly Kov\'acs and Felix Lindner}
\date{}

\begin{document}
\allowdisplaybreaks
\maketitle

\begin{abstract}
We consider autonomous stochastic ordinary differential equations (SDEs) and weak approximations of their solutions for a general class of sufficiently smooth path-dependent functionals $f$. Based on tools from functional It\^o calculus, such as the functional It\^o formula and functional Kolmogorov equation,
we derive a general representation formula for the weak error
$\bE(f(X_T)-f(\tilde X_T))$, where $X_T$ and $\tilde X_T$ are the paths of the solution process
and its approximation
up to time $T$. The functional $f\colon C([0,T],\bR^d)\to \bR$ is assumed to be twice continuously Fréchet differentiable with derivatives of polynomial growth.
The usefulness of the formula is demonstrated in the one dimensional setting
by showing that if the solution to the SDE is approximated via the linearly time-interpolated explicit Euler method, then the rate of weak convergence for sufficiently regular $f$ is $1$.
\end{abstract}
\bigskip
\textbf{Keywords:} Functional It\^o calculus, stochastic differential equation, Euler scheme, weak error, path-dependent functional
\\[\medskipamount]
\textbf{MSC 2010:} 60H10, 60H35, 65C30

\section{Introduction}
\label{sec:Introduction}

Let $(W(t))_{t\geq0}$ be an $m$-dimensional Wiener process and $(X(t))_{t\geq0}$ be the strong solution to a stochastic differential equation (SDE, for short) of the form
\begin{equation}\label{eq:SDE}
\dl X(t) = b(X(t))\,\dl t + \sigma(X(t))\,\dl W(t)
\end{equation}
with initial condition $X(0)=\xi_0\in\bR^d$. The functions $b\colon\bR^d\to\bR^d$ and $\sigma\colon\bR^d\to \bR^{d\times m}$ are assumed to be smooth (i.e., $C^\infty$-functions) such that all derivatives of order $\geq1$ are bounded and $\sigma$ satisfies a non-degeneracy condition, see Section~\ref{sec:preliminaries} for details. Fix $T\in(0,\infty)$ and let $(Y(t))_{t\in[0,T]}$ be a process with continuous sample paths arising from a numerical discretization of \eqref{eq:SDE} which approximates $X$ on $[0,T]$. Let $X_T$ and $Y_T$ denote the $C([0,T],\bR^d)$-valued random variables $\omega\mapsto X(\cdot,\omega)|_{[0,T]}$ and $\omega\mapsto Y(\cdot,\omega)=Y(\cdot,\omega)|_{[0,T]}$, where $X(\cdot,\omega)$ and $Y(\cdot,\omega)$ are the trajectories $t\mapsto X(t,\omega)$ and $t\mapsto Y(t,\omega)$.  In this article, we are interested in analyzing the weak approximation error
\begin{equation}\label{eq:errorYX}
\bE\big(f(Y_T)-f(X_T)\big),
\end{equation}
for sufficiently smooth path-dependent functionals $f\colon C([0,T],\bR^d)\to\bR$. To this end, suppose that we are given a further process $(\tilde X(t))_{t\in[0,T]}$ solving an SDE of the form
\begin{equation}\label{eq:SDEtilde}
\dl \tilde X(t) = \tilde b(t,\tilde X_t)\,\dl t + \tilde \sigma(t,\tilde X_t)\,\dl W(t)
\end{equation}
with initial condition $\tilde X(0)=X(0)=\xi_0\in\bR^d$, where $\tilde b(t,\cdot)\colon C([0,t],\bR^d)\to\bR^d$ and $\tilde\sigma(t,\cdot)\colon C([0,t],\bR^d)\to\bR^{d\times m}$, $t\in[0,T]$, are path-dependent coefficients and $\tilde X_t$ denotes the $C([0,t],\bR^d)$-valued random variable $\omega\mapsto \tilde X(\cdot,\omega)|_{[0,t]}$. If the coefficients $\tilde b$ and $\tilde\sigma$ are chosen in such a way that the error $\bE(f(Y_T)-f(\tilde X_T))$ has a simple structure and can be handled relatively easily, then the problem of analyzing  \eqref{eq:errorYX} essentially reduces to analyzing the weak error
\begin{equation}\label{eq:errortildeXX}
\bE\big(f(\tilde X_T)-f(X_T)\big).
\end{equation}
Our main result, Theorem~\ref{thm:errorRepresentation}, provides a representation formula for the error \eqref{eq:errortildeXX} which is suitable to derive explicit convergence rates for numerical discretization schemes. It is valid under the assumption that $f\colon C([0,T],\bR^d)\to\bR$ is twice continuously Fréchet differentiable and $f$ and its derivatives have at most polynomial growth, $C([0,T],\bR^d)$ being endowed with the uniform norm. The proof is based on tools from functional Itô calculus, such as the functional Itô formula and functional backward Kolmogorov equation, cf.~\cite{BCC16, ConFou10a, ConFou10, ConFou13, ConLu15, Dup09, Fou10}.

As a concrete application, we consider for $d=m=1$ the explicit Euler-Maruyama scheme with maximal step-size $\delta>0$. In order to construct a process $Y$ with computable sample paths, we linearly interpolate the output of the scheme between the nodes. This process, however, does not satisfy an SDE such as \eqref{eq:SDEtilde}. Therefore, we also consider a stochastic interpolation $\tilde{X}$ of the scheme via Brownian bridges which is not feasible for numerical computations but satisfies \eqref{eq:SDEtilde} with suitably chosen coefficients $\tilde{b}$ and $\tilde{\sigma}$. 
Using a Lévy-Ciesielsky type expansion of Brownian motion, we show in Proposition~\ref{prop:yt} that the error $\bE(f(Y_T)-f(\tilde X_T))$ is $O(\delta)$
whenever $f\colon C([0,T],\bR)\to\bR$ is twice continuously Fréchet differentiable and its derivatives have at most polynomial growth. 
For the analysis of the error
$
\bE(f(\tilde X_T)-f(X_T)),
$
we use the error representation formula 
from Theorem~\ref{thm:errorRepresentation}
and show, in Proposition \ref{prop:Euler_error_2}, that it is also $O(\delta)$ if $f\colon C([0,T],\bR)\to\bR$ is four times continuously Fréchet differentiable with derivatives of polynomial growth. As a direct consequence, our main result concerning the linearly interpolated explicit Euler-Mayurama scheme, Theorem \ref{thm:euler}, is that if $f\colon C([0,T],\bR)\to\bR$ is four times continuously Fréchet differentiable and its derivatives have at most polynomial growth, then the weak error $\bE(f(X_T)-f(Y_T))$ is of order $O(\delta)$. The result can be used, for instance, to show that the bias $\operatorname{Cov}(Y(t_1),Y(t_2))-\operatorname{Cov}(X(t_1),X(t_2))$ for the approximation of covariances of the solution process is $O(\delta)$, see Example~\ref{ex:covariance}.

There exists an extensive literature on strong and weak convergence rates of numerical approximations schemes for SDEs, see, e.g., \cite{GraTal13,KP92,MT04} and the references therein. The interplay of strong and weak approximation errors is particularly important for the analysis of multilevel Monte Carlo methods. It is well-known that for various discretization schemes and sufficiently smooth test functions the order of weak convergence exceeds the order of strong convergence and is, in many cases, twice the strong order. However, the weak error analysis of SDEs is often restricted to functionals which only depend on the value of the solutions process at a fixed time, say $T$. Such functionals are of the form $f(X_T)=\varphi(X(T))$ for a function $\varphi\colon\bR^d\to\bR$.
There are not so many publications treating convergence rates of weak approximation errors 
for path-dependent functionals of the solution process as in \eqref{eq:errorYX} and \eqref{eq:errortildeXX}. 
In \cite{GobLab08} Malliavin calculus methods are used to derive estimates for the convergence of the density of the solution to the Euler-Maruyama scheme, leading to $O( \delta)$ weak convergence for a specific class of integral type functionals.
Compositions of smooth functions and non-smooth integral type functionals are treated in \cite{HMN14} for an exact simulation of the solution process at the time discretization points in a one-dimensional setting. Weak convergence rates for Euler-Maruyama approximations of non-smooth path-dependent functionals of solutions to SDEs with irregular drift and constant diffusion coefficient are derived in \cite{NT15}  via a suitable change of measure, the obtained order of convergence being at most $O(\delta^{1/4})$. Weak convergence results for approximations of path-dependent functionals of SDEs without explicit rates of convergence can be found in several articles, e.g., in \cite{BayFri13,SYZ13}. In \cite{ConLu15} the authors use methods from functional Itô calculus to analyze Euler approximations of path-dependent functionals of the form $f(X_T)$ and to derive convergence rates for the corresponding strong error $\bE(|f(X_T)-f(\tilde X_T)|^{2p})$, $p\geq1$.
This list of references is only indicative and we also refer to the references in the mentioned articles.
In this paper, we present a new and general method for the weak error analysis of numerical approximations of a large class of sufficiently smooth, path-dependent functionals of solutions to SDEs of the type \eqref{eq:SDE}. Our approach is based on the  functional Itô calculus as presented in \cite{BCC16, ConFou13, Fou10} and is in a sense a natural, albeit highly nontrivial, generalization of the `classical' approach to the analysis of weak approximation errors based on Itô's formula and backward Kolmogorov equations, cf., e.g., \cite{TalTub90} or \cite[Section~14.1]{KP92}.

Let us remark that weak error estimates are also available for SPDEs, see, e.g., \cite{CJK14,KovLarLin11,KovLarLin13,KovLinSch15,KovPri14b,LinSch13}. In particular,
 path-dependent functionals of solutions to semilinear SPDEs with additive noise are considered in \cite{AndKovLar16,BHS16}. The analysis in \cite{AndKovLar16} is based on Malliavin calculus and applies to certain compositions of smooth functions and integral type functionals. A quite general class of  path-dependent $C^2$-functionals is treated in \cite{BHS16}, based on a second order Taylor expansion of the composition of the test function and an underlying Itô map. A difference to our results (apart from the infinite dimensionality of the state space) is that the analysis in \cite{BHS16} is restricted to spatial discretizations, additive noise, and the test functions are assumed to be bounded.

To present the main idea behind our approach, let $t\geq0$ and for a (deterministic) càdlàg path $x\in D([0,t],\bR^d)$ let the process $X^{t,x}=(X^{t,x}(s))_{s\geq0}$ be defined by
\[X^{t,x}(s):=
\begin{cases}
x(s)&\text{if }s\in[0,t)\\
X^{t,x(t)}(s)&\text{if }s\in[t,\infty)
\end{cases},
\]
where $(X^{t,x(t)}(s))_{s\in[t,\infty)}$ is the strong solution to Eq.~\eqref{eq:SDE} started at time $t$ from $x(t)\in\bR^d$.
For $\epsilon>0$ define a family of functionals $F^\epsilon=(F^\epsilon_t)_{t\in[0,T]}$ by
\begin{equation}\label{eq:defFteps}
F_t^\epsilon(x):=\bE f^\epsilon(X^{t,x}_{T}),\qquad x\in D([0,t],\bR^d),
\end{equation}
where $X^{t,x}_T$ denotes the path of $X^{t,x}$ up to time $T$ and $f^\epsilon$ is a suitably regularized version of the path-dependent functional $f$ such that
$$
\bE\big(f(\tilde X_T)-f(X_T)\big)=\lim_{\epsilon\to0}\bE\big(f^\epsilon(\tilde X_T)-f^\epsilon(X_T)\big).
$$
Then, as we assume that $X(0)=\tilde X(0)=\xi_0\in\bR^d$, it follows that
\begin{align*}
\bE\big(f^\epsilon(\tilde X_T)-f^\epsilon(X_T)\big)
=\bE\big(F^\epsilon_T(\tilde X_T)-F^\epsilon_0(\tilde X_0)\big).
\end{align*}
After proving that $F^\epsilon$ is regular enough in a suitable sense we apply the functional Itô formula from Theorem \ref{thm:functionalIto} to $F^\epsilon_T(\tilde X_T)-F^\epsilon_0(\tilde X_0)$ and use a backward functional Kolmogorov equation from Theorem \ref{thm:functionalKolmogorov} to eliminate a term which cannot be controlled as $\epsilon\to 0$. Finally we arrive at our explicit representation formula for the weak error \eqref{eq:errortildeXX} in terms of $\bE(f^\epsilon(\tilde X_T)-f^\epsilon(X_T))$, stated in Theorem \ref{thm:errorRepresentation}, for $f\colon C([0,T],\bR^d)\to\bR$ twice continuously Fréchet differentiable with at most polynomially growing derivatives:
\begin{equation*}
\begin{aligned}
&\bE\big(f^\epsilon(\tilde X_T)-f^\epsilon(X_T)\big)\\
&= \bE\Bigg(\int_0^T\sum_{j=1}^d\big(\bE\big[Df^\epsilon(X^{t,x}_T)\,(\one_{[t,T]}D^{e_j}\!X^{t,x(t)}_T)\big]\big)\big|_{x=\tilde X_t}\, \big(\tilde b_j(t,\tilde X_t)-b_j(\tilde X(t))\big)\,\dl t \\
&\quad+\frac12\int_0^T\sum_{i,j,k=1}^d\Big\{\Big(\bE\Big[D^2f^\epsilon(X^{t,x}_T)\,\big(\one_{[t,T]}D^{e_i}X^{t,x(t)}_T,\,\one_{[t,T]}D^{e_j}X^{t,x(t)}_T\big)\\
&\quad+Df^\epsilon(X^{t,x})\,\big(\one_{[t,T]}D^{e_i+e_j}X^{t,x(t)}_T\big)\Big]\Big)\Big|_{x=\tilde X_t}\big(\tilde \sigma_{ik}\,\tilde \sigma_{jk}(t,\tilde X_t)-\sigma_{ik}\,\sigma_{jk}(\tilde X(t))\big)\Big\}\,\dl t\Bigg).
\end{aligned}
\end{equation*}
Here, $(e_i)_{i\in\{1,\ldots,d\}}$ is the canonical orthonormal basis of $\bR^d$ and, for a multi-index $\alpha\in\bN_0^d$, $D^{\alpha}X^{t,x(t)}=D^{\alpha}_\xi X^{t,\xi}|_{\xi=x(t)}$ denotes the corresponding partial derivative of the solution process $X^{t,\xi}$ started at time $t$ w.r.t.\ the initial condition $\xi\in\bR^d$, evaluated at $\xi=x(t)$, see Section~\ref{sec:inreg} for details.

The paper is organized as follows. In Section \ref{sec:preliminaries} we introduce some general notation used throughout the article, state the main assumptions on the coefficients in \eqref{eq:SDE} and \eqref{eq:SDEtilde} and also introduce the regularized versions $f^{\epsilon}$, $\epsilon>0$, of a functional $f\colon C([0,T],\bR^d)\to\bR$ via a mollification operator. Section \ref{sec:funIto} contains a short introduction to the notions and notations of the functional It\^o calculus and at the end of the section we also recall the functional It\^o formula as well the functional backward Kolmogorov equation. In Section \ref{sec:inreg} we prove results, crucial for what follows after, concerning the regularity of the solution of \eqref{eq:SDE} with respect to the initial data mainly in the uniform topology. Section \ref{sec:RegularityFFepsilon} is devoted to the study of the regularity of the functional $F^{\epsilon}$ and the explicit computation of its vertical and horizontal derivatives,
so that the functional It\^o formula and the functional backward Kolmogorov equation can be applied; the main findings are summarized in Theorem \ref{thm:regularityFepsilon}. In Section \ref{sec:fk}, using the regularity results from Section \ref{sec:RegularityFFepsilon} and the martingale property of $(F_t^{\epsilon}(X_t))_{t\in [0,T]}$
from Proposition \ref{prop:martingaleProperty}, we show in Corollary \ref{cor:kol} that $F^{\epsilon}$ satisfies a functional backward Kolmogorov equation. Theorem~\ref{thm:errorRepresentation} in Section \ref{sec:errep} contains our main result concerning the representation of the weak error $\bE(f^\epsilon(\tilde X_T)-f^\epsilon(X_T))$. As an important application of Theorem \ref{thm:errorRepresentation},  in Section \ref{sec:euler} we analyse the order of the weak error for the linearly interpolated explicit Euler-Maruyama scheme and the main result here is presented in Theorem \ref{thm:euler}. Finally, in the Appendix, we present a general convergence lemma, Lemma \ref{lem:LpConvImpliesWeakConv}, which is used extensively throughout the paper and also a result from the literature, Lemma \ref{lem:topsup}, concerning the topological support of the distribution $\bP_{X_T}$ of $X_T$ in $C([0,T],\bR^d)$.

\section{Preliminaries}
\label{sec:preliminaries}

In this section we describe some general notation used throughout the article, formulate the precise assumptions on the SDEs \eqref{eq:SDE} and \eqref{eq:SDEtilde} for $X$ and $\tilde X$, and introduce a  mollification operator $M^\epsilon$  that allows us to define suitable smooth approximations $f^\epsilon$ of a
given path-dependent functional $f\colon C([0,T],\bR^d)\to\bR$.
\bigskip

\emph{General notation.}
The natural numbers excluding and including zero are denoted by $\bN=\{1,2,\ldots\}$ and $\bN_0=\{0,1,\ldots\}$, respectively. Norms in finite dimensional real vector spaces are denoted by $|\cdot|$. We usually consider the Euklidean norm, e.g., $|\xi|=\sqrt{\xi_1^2+\ldots+\xi_d^2}$ for a vector $\xi=(\xi_1,\ldots,\xi_d)\in\bR^d$ or $|A|=\sqrt{\sum_{i,j}a_{ij}^2}$ for a matrix $A=(a_{ij})$, but the specific choice of the norm will not be important. The only exception are multi-indices $\alpha=(\alpha_1,\ldots,\alpha_d)\in\bN_0^d$, for which we set $|\alpha|:=\alpha_1+\ldots+\alpha_d$. The canonical orthonormal basis in $\bR^d$ is denoted by $(e_i)_{i\in\{1,\ldots,d\}}$.

By $C([a,b],\bR^d)$ and $D([a,b],\bR^d)$ we denote the spaces of continuous functions and càdlàg (right continuous with left limits) functions defined on an interval $[a,b]$ with values in $\bR^d$, respectively. Both spaces are endowed with the uniform norm, e.g., $\|x\|_{C([a,b];\bR^d)}=\sup_{t\in[a,b]}|x(t)|$. 

For a càdlàg path $x\in D([0,T],\bR)$ and $t\in[0,T]$, we denote by
\[x_t:=x|_{[0,t]}\in D([0,t],\bR^d)\]
the restriction of $x$ to $[0,t]$, whereas $x(t)\in\bR^d$ denotes the value of $x$ at $t$. Consistent with this notation, we will occasionally also write $x_t$ instead of $x$ for a given path $x\in D([0,t],\bR^d)$ in order to indicate the domain of definition. More generally, if $x$ is a càdlàg path defined on an arbitrary interval $I\subset[0,\infty)$ and if $t\in I$, then $x_t:=x|_{I\cap[0,t]}$ denotes the restriction of $x$ to $I\cap[0,t]$. Accordingly, if $Z=(Z(s))_{s\in[a,b]}$ or $Z=(Z(s))_{s\geq a}$ is an $\bR^d$-valued stochastic process with càdlàg paths and if $t\in[a,b]$ or $t\geq a$, we write $Z_t$ for the $D([a,t];\bR^d)$-valued random variable $\omega\mapsto Z(\cdot,\omega)|_{[a,t]}$, where $Z(\cdot,\omega)$ is a trajectory of $Z$. For instance, if $X^{t,\xi}$ is the strong solution to \eqref{eq:SDE} started at time $t\in[0,T]$ from $\xi\in\bR^d$, then $X^{t,\xi}_T$ denotes the $D([t,T];\bR^d)$-valued random variable $\omega\mapsto X^{t,\xi}(\cdot,\omega)|_{[t,T]}$.

Let $(U,\|\cdot\|_U)$ and $(V,\|\cdot\|_V)$ be two normed real vector spaces. We denote by $\bo(U,V)$ the space of bounded linear operators $T\colon U\to V$, endowed with the operator norm $\|T\|_{\bo(U,V)}:=\sup_{\|u\|_U\leq 1}\|Tu\|_V$. For $n\in\bN$, we write $\bo^{(n)}(U,V)$ for the space of bounded $n$-fold multilinear operators $T\colon U^n\to V$, endowed with the norm $\|T\|_{\bo^{(n)}(U,V)}:=\sup_{\|u_1\|_U\leq 1,\ldots,\|u_n\|_U\leq 1}\|T(u_1,\ldots,u_n)\|_V$. If $g\colon U\to V$ is $n$-times Fréchet differentiable, we write $D^ng(u)$ for the $n$-th Fréchet derivative of $g$ at $u\in U$ and consider it as an element of $\bo^{(n)}(U,V)$; by $(D^ng(u))(u_1,\ldots,u_n)\in V$ we denote the evaluation of $D^ng(u)$ at $(u_1,\ldots,u_n)\in U^n$. Specifically, if  $g\colon \bR^d\to V$ is $n$-times Fréchet differentiable and $\alpha\in\bN_0^d$ is a multi-index with $|\alpha|\leq n$, we write $D^\alpha g(\xi):=D_\xi^\alpha g(\xi):=\frac{\partial^{|\alpha|}}{\partial\xi_1^{\alpha_1},\ldots,\partial\xi_d^{\alpha_d}}g(\xi)\in V$ for the corresponding partial derivative at a point $\xi\in\bR^d$. We write $C^n(U,V)$ for the space space of $n$-times continuously Fréchet-differentiable functions from $U$ to $V$, and $C^n_p(U,V)$ is the subspace of $n$-times continuously Fréchet-differentiable functions $g\colon U\to V$ such that $g$ and its derivatives up to order $n$ have at most polynomial growth at infinity, i.e., $C^n_p(U,V)=\{g\in C^n(U,V):\exists\, C,q\geq1 \text{ such that }\|D^ng(u)\|_{\bo^{(n)}(U,V)}\leq C(1+\|u\|_U^q)\text{ for all }u\in U\}$.

Throughout the article, $C\in(0,\infty)$ denotes a finite constant which may change its value with every new appearance.
\bigskip

\emph{Main assumptions.}
Throughout the article, we suppose that the following assumptions hold.
All random variables and stochatic processes are assumed to be defined on a common filtered probability space $(\Omega,\cF,(\cF_t)_{t\geq0},\bP)$ satisfying the usual conditions. The process $(W(t))_{t\geq0}$ is a $\bR^m$-valued Wiener process w.r.t.\ the filtration $(\cF_t)_{t\geq0}$. Concerning the coefficients appearing in the SDEs \eqref{eq:SDE} and \eqref{eq:SDEtilde} for $X$ and $\tilde X$ we assume the following.

\begin{assumption}\label{ass:bsigma}
The functions $b\colon\bR^d\to\bR^d$ and $\sigma\colon\bR^d\to \bR^{d\times m}$ in Eq.~\eqref{eq:SDE} are $C^\infty$-functions such that all derivatives of order $\geq1$ are bounded. There exists a constant $c>0$ such that
$|\sigma(x)y|\geq c|y|$ for all $x\in\bR^d$ and $y\in\bR^m$.
\end{assumption}

Assumption~\ref{ass:bsigma} implies that for every initial condition $\xi\in\bR^d$ and $s\geq0$ there exists a unique stong solution $(X^{s,\xi}(t))_{t\in[s,\infty)}$ to Eq.~\eqref{eq:SDE} starting at time $s$ from $\xi$. By $(X(t))_{t\geq0}=(X^{0,\xi_0}(t))_{t\geq0}$ we denote the solution to \eqref{eq:SDE} starting at time zero in a fixed given starting point $\xi_0\in\bR^d$.

\begin{assumption}\label{ass:bsigmatilde}
The functions $\tilde b(t,\cdot)\colon C([0,t],\bR^d)\to\bR^d$ and $\tilde\sigma(t,\cdot)\colon C([0,t],\bR^d)\to\bR^{d\times m}$, $t\in[0,T]$, in Eq.~\eqref{eq:SDEtilde} are such that
\begin{itemize}
\item
the mapping $(t,x)\mapsto (b(t,x_t),\sigma(t,x_t))$ defined on $[0,T]\times C([0,T],\bR^d)$ is Borel-measurable (recall that $x_t=x|_{[0,t]}$);
\item
there exists a unique strong solution $(\tilde X(t))_{t\in[0,T]}$ to Eq.~\eqref{eq:SDEtilde} starting from $\xi_0$;
\item
the linear growth condition
$|\tilde b(t,x_t)|+|\tilde\sigma(t,x_t)|\leq C (1+\sup_{s\leq t}|x(s)|)$,
$t\in[0,T]$, $x\in C([0,T],\bR^d)$,
is fulfilled (with $C\in(0,\infty)$ independent of $x$ and $s$).
\end{itemize}
\end{assumption}

We note that the boundedness of the derivatives $Db$ and $D\sigma$ and the linear growth assumption on $\tilde b$ and $\tilde\sigma$ imply that
\begin{align}\label{eq:gr}
\bE\big(\sup_{t\in[0,T]}|X(t)|^p\big)+\bE\big(\sup_{t\in[0,T]}|\tilde X(t)|^p\big)<\infty
\end{align}
for all $p\geq1$. This is a consequence of the Burkholder inequality and Gronwall's lemma.
\bigskip

\emph{A mollification operator.}
In order to be able to apply the functional Itô calculus presented in Section~\ref{sec:funIto} to our problem in a convenient way, we associate to every path-dependent functional $f\colon C([0,T],\bR^d)\to\bR$ a family of `regularized' versions $f^\epsilon\colon D([0,T],\bR^d)\to\bR$, $\epsilon>0$, by setting
\begin{equation}\label{eq:deffepsilon}
f^\epsilon := f\circ M^\epsilon.
\end{equation}
Here,
\begin{equation}\label{eq:defMepsilon1}
M^\epsilon\colon D([0,T],\bR^d)\to C^\infty([0,T],\bR^d)
\end{equation}
is the mollification operator defined as follows:
Let $\tilde{\eta}\in C_c^\infty(\bR)$ be a standard mollifier (nonnegative, $\int\tilde{\eta}\dl t=1$, $\supp \tilde{\eta}\subset[-1,1]$) and set $\tilde{\eta}_\epsilon:=(\epsilon/2)^{-1}\tilde{\eta}((\epsilon/2)^{-1}\,\cdot\,)$ as well as $\eta_{\epsilon}(\cdot)=\tilde{\eta}_{\epsilon}(\cdot-\epsilon/2)$. Let $\overline x$ denote the extension of a path $x\in D([0,T],\bR^d)$ to $\bR$ given by \[\overline x(t):=x(0)\one_{(-\infty,0)}(t)+ x(t)\one_{[0,T]}(t)+x(T)\one_{(T,\infty)}(t).\]
Then we set
\begin{equation}\label{eq:defMepsilon2}
M^\epsilon x := (\eta_\epsilon*\overline x)|_{[0,T]},\quad x\in D([0,T],\bR^d),
\end{equation}
where $*$ denotes convolution, i.e.,
$(\eta_\epsilon*\overline x)(t)=\int_\bR\eta_\epsilon(t-s)\overline x(s)\,\dl s$.
Note that, in fact,
\[ (M^\epsilon x)(t)= \int_{t-\epsilon}^t\eta_\epsilon(t-s)\overline x(s)\,\dl s=\int_{-\epsilon}^T\eta_\epsilon(t-s)\overline x(s)\,\dl s,\quad t\in [0,T],\]
and that we have the convergence $M^\epsilon x\xrightarrow{\epsilon\searrow0}x$ in $C([0,T],\bR^d)$ for all $x\in C([0,T],\bR^d)$.

\section{Functional It\^o calculus}
\label{sec:funIto}

In this section we present some of the main notions and results from functional Itô calculus, see \cite{ConFou13} and compare also \cite{BCC16,ConFou10a, ConFou10, ConLu15, Dup09,Fou10}.
On $D([0,t],\bR^d)$ we consider the canonical $\sigma$-algebra $\cB_t$ generated by the cylinder sets of the form $A=\{x\in D([0,t],\bR^d): x(s_1)\in B_1,\ldots, x(s_n)\in B_n\}$, where $0\leq s_1\leq \ldots\leq s_n\leq t$, $B_i\in\cB(\bR^d)$, $i=1,\ldots,n$, and $n\in\bN$. Note that $\cB_t$ coincides with the Borel-$\sigma$-algebra induced by the uniform norm $\|\cdot\|_{D([0,t];\bR^d)}$.

\begin{definition}\label{def:nonanticipativeFunctional}
A {\it non-anticipative functional on $D([0,T],\bR^d)$} is a family $F=(F_t)_{t\in[0,T]}$ of mappings
\[F_t\colon D([0,t],\bR^d)\to\bR,\;x\mapsto F_t(x)\]
such that every $F_t$ is $\cB_t/\cB(\bR)$-measurable.
\end{definition}

We also consider non-anticipative functionals with index set $[0,T)$
as well as $\bR^n$-valued non-anticipative functionals.
These are defined analogously with the obvious modifications.
Recall that for a path $x\in D([0,T],\bR)$ and $t\in[0,T]$, we denote by
$x_t=x|_{[0,t]}\in D([0,t],\bR^d)$
the restriction of $x$ to $[0,t]$ and that, consistent with this notation, we may also write $x_t$ instead of $x$ for a given path $x\in D([0,t],\bR^d)$ in order to indicate the domain of definition.

For $h\geq0$, the {\it horizontal extension} $x_{t,h}\in D([0,t+h],\bR^d)$ of a path $x_t\in D([0,t],\bR^d)$ to $[0,t+h]$ is defined by
\begin{equation*}
x_{t,h}(s):=
\begin{cases}
x_t(s) &\text{if } s\in[0,t)\\
x_t(t) &\text{if } s\in[t,t+h].
\end{cases}
\end{equation*}
For $h\in\bR^d$, the {\it vertical perturbation} $x_t^h\in D([0,t],\bR^d)$ of a path $x_t\in D([0,t],\bR^d)$ is defined by
\begin{equation*}
x_t^h(s):=
\begin{cases}
x_t(s) &\text{if } s\in[0,t)\\
x_t(t)+h &\text{if } s=t.
\end{cases}
\end{equation*}

\begin{definition}
Let $F=(F_t)_{t\in[0,T]}$ be a non-anticipative functional on $D([0,T],\bR^d)$.
\begin{itemize}
\item[(i)]
For $t\in[0,T)$ and $x\in D([0,t],\bR^d)$, the {\it horizontal derivative of $F$ at $x$} is defined as
\begin{equation}\label{eq:horizontalDer}
\cD_t F(x):=\lim_{h\searrow0}\frac{F_{t+h}(x_{t,h})-F_t(x_t)}h
\end{equation}
provided that the limit exists.
If \eqref{eq:horizontalDer} is defined for all $t\in[0,T)$ and $x\in D([0,t],\bR^d)$, then $F$ is called {\it horizontally differentiable}. In this case, the mappings
\begin{equation*}
\cD_t F:D([0,t],\bR^d)\to\bR,\; x\mapsto \cD_t F(x),\qquad t\in[0,T),
\end{equation*}
define a non-anticipative functional $\cD F=(\cD_t F)_{t\in[0,T)}$, the {\it horizontal derivative of~$F$}.
\item[(ii)]
For $t\in[0,T]$ and $x\in D([0,t],\bR^d)$, the {\it vertical derivative of $F$ at $x$} is defined as
\begin{equation}\label{eq:verticalDer}
\nabla_x F_t(x):=\lim_{h\to 0}\Big(\frac{F_t(x_t^{he_1})-F_t(x_t)}h,\ldots,\frac{F_t(x_t^{he_d})-F_t(x_t)}h\Big)
\end{equation}
provided that the limit exists, where $(e_j)_{j=1,\ldots,d}$ is the canonical orthonormal basis in $\bR^d$.
If \eqref{eq:verticalDer} is defined for all $t\in[0,T]$ and $x\in D([0,t],\bR^d)$, then $F$ is called {\it vertically differentiable}. In this case, the mappings
\begin{equation*}
\nabla_x F_t:D([0,t],\bR^d)\to\bR^d,\; x\mapsto \nabla_x F_t(x),\qquad t\in[0,T],
\end{equation*}
define a non-anticipative functional $\nabla_x F=(\nabla_x F_t)_{t\in[0,T]}$, the {\it vertical derivative of~$F$}.
\end{itemize}
\end{definition}

In order to introduce a proper notion of (left-)continuity for non-anticipative functionals, one considers the following distance between two paths which are possibly defined on different time intervals.
For $t,\;t'\in[0,T]$, $x\in D([0,t],\bR^d)$ and $x'\in D([0,t'],\bR^d)$ we define
\begin{equation*}
d_\infty(x,x'):=|t-t'|+\sup_{s\in [0,T]}\big|x_{t,T-t}(s)-x'_{t',T-t'}(s)\big|.
\end{equation*}
We remark that $d_\infty$ is a metric on the set
\begin{equation*}
\Lambda:=\bigcup_{t\in[0,T]} D([0,t],\bR^d).
\end{equation*}

\begin{definition}\label{def:classes}
Let $F=(F_t)_{t\in[0,T]}$ be a non-anticipative functional on $D([0,T],\bR^d)$.
\begin{itemize}
\item[(i)]
$F$ is \emph{continuous at fixed times} if, for all $t\in[0,T]$, the mapping $F_t:D([0,t],\bR^d)\to\bR$ is continuous w.r.t.\ the uniform norm $\nnrm{\cdot}{D([0,t],\bR^d)}$.
\item[(ii)]
$F$ is \emph{continuous} if the mapping $\Lambda\ni x_t\mapsto F_t(x_t)\in\bR$ is continuous w.r.t.\ the metric $d_\infty$ on $\Lambda$, i.e., if
\begin{equation*}
\begin{aligned}
\forall\,t\in[0,T]\;\forall\,x\in D(&[0,t],\bR^d)\;\forall\,\varepsilon>0\;\exists\,\delta>0\;\forall\,t'\in[0,T]\;\forall\,x'\in D([0,t'],\bR^d)\;\\
&\big(d_\infty(x,x')<\delta \,\Rightarrow\, |F_t(x)-F_{t'}(x')|<\varepsilon\big).
\end{aligned}
\end{equation*}
The class of continuous non-anticipative functionals is denoted by $\bC^{0,0}([0,T])$.
\item[(iii)]
$F$ is \emph{left-continuous} if
\begin{equation*}
\begin{aligned}
\forall\,t\in[0,T]\;\forall\,x\in D(&[0,t],\bR^d)\;\forall\,\varepsilon>0\;\exists\,\delta>0\;\forall\,t'\in[0,t]\;\forall\,x'\in D([0,t'],\bR^d)\;\\
&\big(d_\infty(x,x')<\delta \,\Rightarrow\, |F_t(x)-F_{t'}(x')|<\varepsilon\big).
\end{aligned}
\end{equation*}
The class of left-continuous non-anticipative functionals is denoted by $\bC^{0,0}_l([0,T])$.
\item[(iv)]
F is \emph{boundedness-preserving} if
\begin{equation*}
\begin{aligned}
\forall\,R>\,&0\;\exists\,C>0\;\forall\,t\in[0,T]\;\forall\,x\in D([0,t],\bR^d)\\
&\big(\sup_{s\in[0,t]}|x(s)|\leq R \,\Rightarrow\, |F_t(x)|\leq C\big).
\end{aligned}
\end{equation*}
The class of boundedness-preserving non-anticipative functionals is denoted by $\bB([0,T])$.
\end{itemize}
\end{definition}

We will also use the above notions for $\bR^n$-valued non-anticipative functionals;
the corresponding definitions are analogous with the obvious modifications.

\begin{definition}\label{def:regularFunctionals}
For $k\in\bN$, we denote by $\bC^{1,k}_b([0,T])$ be the class of all left-continuous, boundedness-preserving, non-anticipative functionals $F=(F_t)_{t\in[0,T]}\in\bC_l^{0,0}([0,T])\cap\bB([0,T])$ such that
\begin{itemize}
\item $F$ is horizontally differentiable, the horizontal derivative $\cD F=(\cD_t F)_{t\in[0,T)}$ is continuous at fixed times, and the extension $(\cD_t F)_{t\in[0,T]}$ of $(\cD_t F)_{t\in[0,T)}$ by zero belongs to the class $\bB([0,T])$,
\item $F$ is $k$ times vertically differentiable with $\nabla^j_x F=(\nabla^j_x F_t)_{t\in[0,T]}\in\bC_l^{0,0}([0,T])\cap\bB([0,T])$ for all $j=1,\ldots,k$.
\end{itemize}
\end{definition}

\begin{remark}\label{rem:defBoundednessPreserving}
We remark that in \cite{ConFou13}, our main reference for functional Itô calculus, the slightly different class of boundedness-preserving functionals $\bB([0,T))$ with index set $[0,T)$ is considered instead of the class $\bB([0,T])$ introduced in Definition~\ref{def:classes} above. In contrast to the latter the boundedness assumption for functionals in the former class in not uniform in time. Our definition corresponds to the one in \cite{ConLu15}. Similarly, the class $\bC^{1,k}_b([0,T))$ of  regular and boundedness-preserving functionals considered in \cite{ConFou13} differs from the class $\bC^{1,k}_b([0,T])$ introduced in Definition~\ref{def:regularFunctionals} above. As a consequence, the choice $t=T$ is admissible in the functional Itô formula below.
\end{remark}

Next, we state a functional version of Itô's formula, compare \cite[Theorem~4.1]{ConFou13} or \cite{BCC16, ConFou10, ConFou10a, ConLu15,Dup09, Fou10}.

\begin{theorem}\label{thm:functionalIto}
Let $Y=(Y(t))_{t\in[0,T]}$ be an $\bR^d$-valued continuous semimartingale defined on $(\Omega,\cF,\bP)$ and $F=(F_t)_{t\in[0,T]}$ a non-anticipative functional belonging to the class $\bC^{1,2}_b([0,T])$. Then, for all $t\in[0,T]$,
\begin{align*}
F_t(Y_t) &= F_0(Y_0) + \int_0^t\cD_sF(Y_s)\,\dl s + \int_0^t\nabla_xF_s(Y_s)\; \dl Y(s) +\frac12\int_0^t\tr \big(\nabla_x^2 F_s(Y_s)\;\dl [Y](s)\big).
\end{align*}
\end{theorem}

The following result concerning functional Kolmogorov equations is taken from \cite[Theorem~3.7]{Fou10}, compare also \cite[Chapter~8]{BCC16}.
\begin{theorem}\label{thm:functionalKolmogorov}
Let $X=(X_t)_{t\geq0}$ be the solution to Eq.~\eqref{eq:SDE} and $F\in \bC^{1,2}_{\operatorname b}([0,T])$. The process $(F_t(X_t))_{t\in[0,T]}$ is a martingale w.r.t.\ $(\cF_t)_ {t\in[0,T]}$ if, and only if, $F$ satisfies the functional partial differential equation
\begin{equation*}
\cD_t F(x_t)=-b(x(t))\nabla_x F_t(x_t)-\frac12\tr\big(\nabla_x^2 F_t(x_t)\,\sigma(x(t))\,\sigma^\top(x(t))\big)
\end{equation*}
for all $t\in(0,T)$ and all $x\in C([0,T],\bR^d)$ belonging to the topological support of $\bP_{X_T}$ in $(C([0,T],\bR^d),\nnrm{\cdot}\infty)$.
\end{theorem}

\section{Smoothness with respect to the initial condition}
\label{sec:inreg}
%
%

Here we collect and derive several auxiliary results concerning the regularity of the solution to Eq.~\eqref{eq:SDE} with respect to the initial condition. They are crucial for the regularity properties of the functional $F^\epsilon$ and the explicit representation of its derivatives as proved in Section~\ref{sec:RegularityFFepsilon}. 

Recall that $X^{s,\xi}=(X^{s,\xi}(t))_{t\in[s,\infty)}$ denotes the solution to \eqref{eq:SDE} started at time $s\geq0$ from $\xi\in\bR^d$. Given $p\geq1$ and a random variable $Y\in L^p(\Omega,\mathcal F_s,\bP;\bR^d)$, we use the analogue notation $X^{s,Y}=(X^{s,Y}(t))_{t\in[s,\infty)}$ for the solution to \eqref{eq:SDE} started at time time $s\geq0$ with initial condition $X^{s,Y}(s)=Y$. For $s\in[0,T]$ and $Y, Z\in L^p(\Omega,\cF_s,\bP;\bR^d)$, the Burkholder inequality and Gronwall's lemma then yield the standard estimates
\begin{align}
\bE\big(\sup_{t\in[s,T]}|X^{s,Y}(t)|^p\big) &\leq C\big(1+\bE(|Y|^p)\big),\label{eq:standardEstInitialCond1}\\
\bE\big(\sup_{t\in[s,T]}|X^{s,Y}(t)-X^{s,Z}(t)|^p\big) &\leq C\, \bE(|Y-Z|^p),
\label{eq:standardEstInitialCond2}
\end{align}
where $C=C_{p,T,\sigma,b}\in(0,\infty)$ does not depend on $Y$, $Z$ or $s$.
Moreover, under our assumptions on $\sigma$ and $b$ it is well known that, for fixed $s\geq0$, the random field $(X^{s,\xi}(t))_{t\in[s,\infty),\,\xi\in\bR^d}$ has a modification such that, for $\bP$-almost all $\omega\in\Omega$,
the mapping
\[[s,\infty)\times\bR^d\ni(t,\xi)\mapsto X^{s,\xi}(t,\omega)\in\bR^d\]
is continuous and for all $t\in[s,\infty)$ the mapping
\[\bR^d\ni\xi\mapsto X^{s,\xi}(t,\omega)\in\bR^d\]
is infinitely often differentiable, see, e.g.\ \cite[Section V.2]{IkeWat89}. In particular, every continuous modification of $(X^{s,\xi}(t))_{\in[s,\infty),\,\xi\in\bR^d}$ satisfies this property of smoothness w.r.t.\ the initial condition.
The reasoning in the proof of Proposition V.2.2 in \cite{IkeWat89} and the time-homogeneity of Eq.~\eqref{eq:SDE} also yield that for every multi-index $\alpha\in\bN_0^d$, $p\geq1$ and all bounded sets $\cO\subset\bR^d$ the partial derivatives $D^\alpha X^{s,\xi}(t,\omega)=D^\alpha_\xi X^{s,\xi}(t,\omega)$ satisfy the estimate 
\begin{equation}\label{eq:unifLpEstDalphaX}
\sup_{s\in[0,T]}\bE\big(\sup_{\xi\in \cO,\,t\in[s,T]}|D^\alpha X^{s,\xi}(t)|^p\big) <\infty.
\end{equation}

For the proof of our error expansion we need to check that the partial derivatives $D^\alpha_\xi X^{s,\xi}(t,\omega)$, $\alpha\in\bN_0^d$, can be taken uniformly with respect to \ $t\in[s,T]$ and that the $L^p(\Omega;C([s,T],\bR^d))$-norms of these derivatives are bounded in $\xi\in\bR^d$.
As already mentioned in Section~\ref{sec:preliminaries}, we use the notation $D^\alpha$ also for the partial derivatives of general Banach space-valued functions. That is, if $B$ is a Banach space, $g:\bR^d\to B$ a sufficiently often (Fréchet-)differentiable function, $\xi=(\xi_1,\ldots,\xi_d)\in\bR^d$ and $\alpha=(\alpha_1,\ldots,\alpha_d)\in\bN_0^d$, then
\[D^\alpha g(\xi):=D_\xi^\alpha g(\xi):=\frac{\partial^{|\alpha|}}{\partial\xi_1^{\alpha_1},\ldots,\partial\xi_d^{\alpha_d}}g(\xi)\;\in B\]
denotes the corresponding partial derivative of order $|\alpha|=\alpha_1+\ldots+\alpha_d$ of $f$ at $\xi$. In the sequel, we use this notation both in the case $B=\bR^d$ and $g(\xi)=X^{s,\xi}(t,\omega)$ with fixed $t\geq s$ and in the case $B=C([s,T],\bR^d)$ and $g(\xi)=X^{s,\xi}_T(\cdot,\omega)$.

\begin{theorem}\label{thm:regularityInitialCond}
For $s\in[0,T]$ fix a continuous modification of $(X^{s,\xi}(t))_{t\in[s,T],\,\xi\in\bR^d}$.
\begin{itemize}
\item[(i)]
For $\bP$-almost all $\omega\in\Omega$, the mapping
\begin{equation*}
\bR^d\ni\xi\mapsto X^{s,\xi}_T(\cdot,\omega)\in C([s,T],\bR^d)
\end{equation*}
is infinitely often (Fréchet-)differentiable.
In particular, the partial derivatives
\[D^\alpha X^{s,\xi}_T(\cdot,\omega)=D_\xi^\alpha X^{s,\xi}_T(\cdot,\omega),\;\alpha\in\bN_0^d,\;\xi\in\bR^d,\]
exist as $C([s,T],\bR^d)$-limits of the corresponding $C([s,T],\bR^d)$-valued difference quotients.
\item[(ii)]
For all $\alpha\in\bN_0^d\setminus\{0\}$ and $p\in[1,\infty)$ we have
\begin{equation}\label{eq:LpEstDalphaX}
\sup_{s\in[0,T],\,\xi\in\bR^d}\bE\big(\|D^\alpha X^{s,\xi}_T\|_{C([s,T],\bR^d)}^p\big)<\infty.
\end{equation}
\end{itemize}
\end{theorem}

In the proof of Theorem~\ref{thm:regularityInitialCond} and in Corollary \ref{cor:chainRuleDalphaX} we will encounter certain higher order chain rules  of Faà di Bruno type, for which the following notation will be convenient.
\begin{notation}\label{not:FaadiBruno}
For a given multi-index $\alpha\in\bN_0^d\setminus\{0\}$ we denote by
$\Pi(\{1,\ldots,|\alpha|\})\subset\cP(\cP(\{1,\ldots,|\alpha|\}))$ the set of all partitions of the set $\{1,\ldots,|\alpha|\}$. By $|\pi|$ we denote the size of a partition $\pi\in\Pi(\{1,\ldots,|\alpha|\})$, i.e., the number of subsets of $\{1,\ldots,|\alpha|\}$ contained in $\pi$. The disjoint subsets of $\{1,\ldots,|\alpha|\}$ contained in a partition $\pi\in\Pi(\{1,\ldots,|\alpha|\})$ are denoted by $\pi_1,\ldots,\pi_{|\pi|}$, i.e., $\pi=\{\pi_1,\ldots,\pi_{|\pi|}\}$. Finally, we associate to every subset $S\subset\{1,\ldots,|\alpha|\}$ a multi-index $\alpha_S\in\bN_0^d$ by setting $\alpha_S:=|\{k\in S:1\leq k\leq \alpha_1\}|\,e_1+\sum_{j=2}^d |\{k\in S:\sum_{i=1}^{j-1}\alpha_i< k\leq \sum_{i=1}^{j}\alpha_i\}|\,e_j$.
\end{notation}

\begin{proof}[Proof of Theorem~\ref{thm:regularityInitialCond}]

(i) The proof of Proposition V.2.2 in \cite{IkeWat89} implies that, for almost all $\omega\in\Omega$, the mappings $\bR^d\ni\xi\mapsto X^{s,\xi}(t,\omega)\in\bR^d$, $t\geq s$, are infinitely often differentiable, the mappings $[s,\infty)\times\bR^d\ni(t,\xi)\mapsto D^\alpha X^{s,\xi}(t,\omega)\in\bR^d$, $\alpha\in\bN_0^d$, are continuous and, due to \eqref{eq:unifLpEstDalphaX},
\begin{equation}\label{eq:Xxidiffble1}
\sup_{\,\xi\in\cO,\,t\in[s,T]}|D^\alpha X^{s,\xi}(t,\omega)|<\infty
\end{equation}
for all bounded domains $\cO\subset\bR^d$ and $\alpha\in\bN_0^d$. Fix such an $\omega\in\Omega$ and let $(e_i)_{i=1,\ldots,d}$ be the canonical orthonormal basis of $\bR^d$. By Taylor's formula, for $h>0$,
\begin{equation}\label{eq:Xxidiffble2}
\begin{aligned}
\sup_{t\in[s,T]}\Bigg|\frac{D^\alpha X^{s,\xi+h e_i}(t,\omega)-D^\alpha X^{s,\xi}(t,\omega)}h& - D^{\alpha+e_i} X^{s,\xi}(t,\omega)\Bigg|\\
&\leq \frac h2\sup_{\substack{t\in[s,T]\\ \xi'\in[\xi,\xi+h e_i]}}\big| D^{\alpha+2e_i} X^{s,\xi'}(t,\omega)\big|.
\end{aligned}
\end{equation}
Combining \eqref{eq:Xxidiffble1} with $\alpha=2 e_i$ and \eqref{eq:Xxidiffble2} with $\alpha=0$ and using the continuity of the mappings $[s,T]\times\bR^d\ni(t,\xi)\to D^{e_i} X^{s,\xi}(t,\omega)\in \bR^d$, $i\in\{1,\ldots,d\}$,
one obtains the (Fréchet-)differentiability of $\bR^d\ni\xi\mapsto X^{s,\xi}_T(\cdot,\omega)\in C([s,T],\bR^d)$ and the identity
\begin{equation}\label{eq:Xxidiffble3}
D^{\alpha} X^{s,\xi}(t,\omega)=(D^{\alpha} X^{s,\xi}_T(\cdot,\omega))(t)
\end{equation}
for all $\xi\in\bR^d$, $t\in[s,T]$ and $\alpha\in\bN_0^d$ with $|\alpha|=1$. In \eqref{eq:Xxidiffble3}, we have a derivative of the function $\bR^d\ni\xi\mapsto X^{s,\xi}(t,\omega)\in\bR^d$ on the left hand side and a derivative of the function $\bR^d\ni\xi\mapsto X^{s,\xi}_T(\cdot,\omega)\in C([s,T],\bR^d)$ on the right hand side. By repeating this argument for the higher derivatives we finish the proof of (i) via induction over $|\alpha|$.

(ii)
For a better readability, we fix $s=0$ for a moment and omit the explicit notation of the initial condition by writing $X(t)$ instead of $X^{0,\xi}(t)$. The proofs of Propositions V.2.1 and V.2.2 in \cite{IkeWat89} imply that, for $\alpha\in\bN_0^d$ with $|\alpha|=1$, the $\bR^d$-valued process $(D^\alpha X(t))_{t\geq0}$ is the solution to the SDE
\begin{equation*}
D^\alpha X(t)=\alpha +\sum_{\nu=1}^m\int_0^t D\sigma_\nu(X(s))\,D^\alpha X(s)\,\dl W_\nu(s) +\int_0^t Db(X(s))\,D^\alpha X(s)\,\dl s.
\end{equation*}
Here we denote for $x\in\bR^d$ by $\sigma_\nu(x)\in\bR^d$ be the $\nu$-th column vector of $\sigma(x)\in\bR^{d\times m}$, $D\sigma_\nu:\bR^d\to\bR^{d\times d}$ and $Db:\bR^d\to\bR^{d\times d}$ are the (total) derivatives of $\sigma_\nu:\bR^d\to\bR^d$ and $b:\bR^d\to\bR^d$, and $W_\nu$ is the $\nu$-th component of $W$. Using the Burkholder inequality we obtain for all $p\geq2$ and $t\in[0,T]$
\begin{equation}\label{eq:Xxidiffble4}
\begin{aligned}
\bE\big(\sup_{r\in[0,t]}|D^\alpha X(r)|^p\big)
&\leq C_{p,T}\Big(1+\bE\int_0^t\big(\sum_{\nu=1}^m|D\sigma_\nu(X(s))\,D^\alpha X(s)|^2\big)^{\frac p2}\,\dl s \\
&\quad +\bE\int_0^t|Db(X(s))\,D^\alpha X(s)|^p\dl s\Big)\\
&\leq C_{p,T,\sigma,b}\Big(1+\int_0^t\bE\big(\sup_{r\in[0,s]}|D^\alpha X(r)|^p\big)\,\dl s\Big),
\end{aligned}
\end{equation}
where the constant $C_{p,T,\sigma,b}\in(0,\infty)$ does not depend on the initial condition $\xi\in\bR^d$. Thus, Gronwall's lemma implies
\begin{equation}\label{eq:Xxidiffble5}
\bE\big(\|D^\alpha X_T\|_{C([0,T],\bR^d)}^p\big)=\bE\big(\sup_{r\in[0,T]}|D^\alpha X(r)|^p\big)
\leq C_{p,T,\sigma,b}\,\exp(C_{p,T,\sigma,b}\,T)
\end{equation}
with the constant $C_{p,T,\sigma,b}$ from \eqref{eq:Xxidiffble4}. Taking into account the time-homogeneity of Eq.~\eqref{eq:SDE} this proves the assertion for $|\alpha|=1$.


For general $\alpha\in\bN_0^d\setminus\{0\}$, the proofs of Propositions V.2.1 and V.2.2 in \cite{IkeWat89} imply that the $\bR^d$-valued process $(D^\alpha X(t))_{t\geq0}$ is the solution to the SDE
\begin{equation}\label{eq:SDE_DalphaX}
\begin{aligned}
D^\alpha X(t)
&=\sum_{\pi\in\Pi(\{1,\ldots,|\alpha|\})}\Bigg\{\sum_{\nu=1}^m\int_0^t D^{|\pi|}\sigma_\nu(X(s))\big(D^{\alpha_{\pi_1}} X(s),\ldots,D^{\alpha_{\pi_{|\pi|}}} X(s)\big)\,\dl W_\nu(s)\\
&\quad+\int_0^t D^{|\pi|}b(X(s))\big(D^{\alpha_{\pi_1}} X(s),\ldots,D^{\alpha_{\pi_{|\pi|}}} X(s)\big)\,\dl s\Bigg\},
\end{aligned}
\end{equation}
where we use Notation~\ref{not:FaadiBruno} and where, for $n=1,\ldots,|\alpha|$, $D^n\sigma_\nu$ and $D^nb$ are the $n$-th total derivatives of $\sigma_\nu:\bR^d\to\bR^d$ and $b:\bR^d\to\bR^d$, considered as functions with values the space of $n$-fold multilinear mappings from $(\bR^d)^n$ to $\bR^d$.
Using \eqref{eq:SDE_DalphaX}, the proof is finished via induction over $|\alpha|$ by arguing similarly as in \eqref{eq:Xxidiffble4} and \eqref{eq:Xxidiffble5} and applying the respective estimates for $\bE\big(\|D^\beta X_T\|_{C([0,T],\bR^d)}^q\big)$, $q\geq 2$, $\beta\in\bN_0^d$ with $|\beta|<|\alpha|$. Passing from $s=0$ to general $s\in[0,T]$ is no problem due to the time-homogeneity of Eq.~\eqref{eq:SDE}.
\end{proof}

In the sequel, we always consider continuous modifications of the random fields $(X^{s,\xi}(t))_{t\in[s,T],\xi\in\bR^d}$, $s\in[0,T]$.

\begin{remark}\label{rem:FrechetDerXsxi}
For $n\in\bN$ and $\omega\in\Omega$ as in Theorem~\ref{thm:regularityInitialCond}~(i), we consider the $n$-th Fréchet derivative of the mapping $\bR^d\ni\xi\mapsto X^{s,\xi}_T(\cdot,\omega)\in C([s,T],\bR^d)$ in $\xi_0\in\bR^d$ as usual as an $n$-fold multilinear mapping from $(\bR^d)^n$ to $C([s,T],\bR^d)$,
\[D^n X^{s,\xi_0}_T(\cdot,\omega)\colon(\bR^d)^n\to C([s,T],\bR^d).\]
Just as in standard calculus one sees that it is given by
\begin{align*}
D^n X^{s,\xi_0}_T(\cdot,\omega)(\eta_1,\ldots,\eta_n)&=\sum_{\substack{\alpha\in\bN_0^d\\|\alpha|= n}}\eta_{1,1}\ldots\eta_{\alpha_1,1}\,\eta_{\alpha_1+1,2}\ldots\eta_{\alpha_1+\alpha_2,2} \ldots\\
&\quad \ldots\eta_{\alpha_1+\ldots+\alpha_{d-1}+1,d}\ldots\eta_{n,d}\, D^\alpha X^{s,\xi_0}_T(\cdot,\omega),
\end{align*}
where $\eta_j=(\eta_{j,1},\ldots,\eta_{j,d})\in\bR^d$, $j=1,\ldots,n$.
\end{remark}


\begin{notation}
Given a $\bR^d$-valued random variable $Y$ we set
\begin{equation}\label{eq:notationDalphasY}
D^\alpha X^{s,Y}_T(\cdot,\omega):= D^\alpha X^{s,Y(\omega)}_T(\cdot,\omega)
=(D^\alpha_\xi X^{s,\xi}_T(\cdot,\omega))|_{\xi=Y(\omega)}\;\in C([s,T],\bR^d)
\end{equation}
for $s\in[0,T]$, $\alpha\in\bN^d_0\setminus\{0\}$ and (almost all) $\omega\in\Omega$. We consider  $D^\alpha X_T^{s,Y}$ optionally as a $\bR^d$-valued process $D^\alpha X_T^{s,Y}=(D^\alpha X_T^{s,Y}(t))_{t\in[s,T]}$ or as a $C([s,T],\bR^d)$-valued random variable,
\[D^\alpha X^{s,Y}_T\colon \Omega\to C([s,T],\bR^d),\;\omega\mapsto D^\alpha X^{s,Y}_T(\omega):=D^\alpha X^{s,Y}_T(\cdot,\omega).\]
 We use the analogue notation for the $n$-th Fréchet derivatives of $\xi\mapsto X^{s,\xi}_T(\cdot,\omega)$ evaluated at $\xi=Y(\omega)$,
\begin{equation}\label{eq:notationDnY}
D^n X^{s,Y}_T(\cdot,\omega):= D^n X^{s,Y(\omega)}_T(\cdot,\omega)
=(D^n_\xi X^{s,\xi}_T(\cdot,\omega))|_{\xi=Y(\omega)}\;\in \bo^{(n)}(\bR^d,C([s,T],\bR^d)),
\end{equation}
where $\bo^{(n)}(\bR^d,C([s,T],\bR^d))$ is the space of bounded, $n$-fold multilinear mappings from $(\bR^d)^n$ to $C([s,T],\bR^d)$.
\end{notation}

Note that the notation \eqref{eq:notationDalphasY} is consistent with our notation $X^{s,Y}=(X^{s,Y}(t))_{t\geq s}$ for the solution of \eqref{eq:SDE} started at time $s$ with $\cF_s$-measurable initial condition $Y$, since
$X^{s,Y}_T(\cdot,\omega)=X^{s,Y(\omega)}_T(\cdot,\omega)$ for almost all $\omega\in\Omega$.

\begin{corollary}\label{cor:ContDalphaX}
Let $s\in[0,T]$ and $Y,\,Y_n$, $n\in\bN$,
be $\cF_s$-measurable, $\bR^d$-valued random variables
such that $Y_n\xrightarrow{n\to\infty}Y$ $\bP$-almost surely. Then, for all $\alpha\in\bN_0^d\setminus\{0\}$ and $p\geq1$,
\[D^\alpha X^{s,Y_n}_T\xrightarrow{n\to\infty}D^\alpha X^{s,Y}_T\quad\text{ in }L^p(\Omega;C([s,T],\bR^d)).\]
\end{corollary}

\begin{proof}
Using standard properties of conditional expectations, we have
\begin{equation*}
\begin{aligned}
\bE\Big(\big\| D^\alpha X^{s,Y}_T-D^\alpha X^{s,Y_n}_T\big\|_{C([s,T],\bR^d)}^p\Big)
&= \bE\Big(\bE\Big(\big\| D^\alpha X^{s,Y}_T-D^\alpha X^{s,Y_n}_T\big\|_{C([s,T],\bR^d)}^p\,\Big|\,\cF_s\Big) \Big)\\
&= \bE\Big(\bE\Big(\big\| D^\alpha X^{s,\xi}_T-D^\alpha X^{s,\eta}_T\big\|_{C([s,T],
\bR^d)}^p\Big)\Big|_{(\xi,\eta)=(Y,Y_n)}\Big).
\end{aligned}
\end{equation*}
Now the assertion follows from the continuity of the mapping $\bR^d\ni\xi \mapsto D^\alpha X^{s,\xi}_T\in C([s,T],\bR^d)$ asserted by Theorem~\ref{thm:regularityInitialCond}(i), the estimates \eqref{eq:unifLpEstDalphaX} and \eqref{eq:LpEstDalphaX}, and two applications of the dominated convergence theorem.
\end{proof}

\begin{corollary}\label{cor:chainRuleDalphaX}
Let $0\leq s\leq t\leq T$, $\xi\in\bR^d$, $\alpha\in\bN^d_0\setminus\{0\}$ and denote by $D^\alpha X^{s,\xi}|_{[t,T]}$ the $C([t,T];\bR^d)$-valued random variable $\omega\mapsto (D^\alpha X^{s,\xi}(\cdot,\omega))|_{[t,T]}$.
\begin{itemize}
\item[(i)]
If $|\alpha|=1$, then
\begin{equation*}
D^\alpha X^{s,\xi}|_{[t,T]}=D X^{t,X^{s,\xi}(t)}_T D^\alpha X^{s,\xi}(t)
\end{equation*}
$\bP$-almost surely in $C([t,T],\bR^d)$.
(Note that
the random variable $D X^{t,X^{s,\xi}(t)}_T$ takes values in $\bo(\bR^d,C([t,T],\bR^d))$ and $D^\alpha X^{s,\xi}(t)$ takes values in $\bR^d$.)
\item[(ii)]
For general $\alpha\in\bN^d_0\setminus\{0\}$ we have
\begin{equation*}
D^\alpha X^{s,\xi}|_{[t,T]}
=\sum_{\pi\in\Pi(\{1,\ldots,|\alpha|\})}D^{|\pi|} X^{t,X^{s,\xi}(t)}_T \big(D^{\alpha_{\pi_1}} X^{s,\xi}(t),\ldots,D^{\alpha_{\pi_{|\pi|}}} X^{s,\xi}(t)\big)
\end{equation*}
$\bP$-almost surely in $C([t,T],\bR^d)$, where we use Notation~\ref{not:FaadiBruno}.
(Note that
the random variable $D^{|\pi|} X^{t,X^{s,\xi}(t)}_T$ takes values in $\bo^{(|\pi|)}(\bR^d,C([t,T],\bR^d))$ and the random variables $D^{\alpha_{\pi_1}} X^{s,\xi}(t),\ldots,D^{\alpha_{\pi_{|\pi|}}} X^{s,\xi}(t)$ take values in $\bR^d$.)
\end{itemize}
\end{corollary}

\begin{proof}
\begin{enumerate}[(i)]
\item
If $|\alpha|=1$ we have $\alpha=e_i$ for some $i\in\{1,\ldots,d\}$. By Theorem~\ref{thm:regularityInitialCond}(i) we know that for almost all $\omega\in\Omega$ the derivative $D^{e_i}X^{s,\xi}(\cdot,\omega)_T=D^{e_i}_\xi X^{s,\xi}(\cdot,\omega)_T$ exists as a $C([s,T],\bR^d)$-limit of the corresponding difference quotient. Let $(h_n)_{n\in\bN}$ be a sequence of positive numbers decreasing to zero. Then, $\bP$-almost surely,
\begin{align*}
D^{e_i}X^{s,\xi}|_{[t,T]}=C([t,T],\bR^d)\text{-}\lim_{n\to\infty}\frac{X^{s,\xi+h_ne_i}|_{[t,T]}-X^{s,\xi}|_{[t,T]}}{h_n}.
\end{align*}
As a consequence of the unique solvability of Eq.~\eqref{eq:SDE}, we  have the identities
\[X^{s,\xi+h_ne_i}|_{[t,T]}=X^{t,X^{s,\xi+h_ne_i}(t)}_T \quad\text{ and }\quad X^{s,\xi}|_{[t,T]}=X^{t,X^{s,\xi}(t)}_T\]
holding  $\bP$-almost surely in $C([t,T],\bR^d)$. Further, recall that $X^{t,X^{s,\xi+h_ne_i}(t)}_T(\cdot,\omega) =X^{t,X^{s,\xi+h_ne_i}(t,\omega)}_T(\cdot,\omega)$ and $X^{t,X^{s,\xi}(t)}_T(\cdot,\omega) =X^{t,X^{s,\xi}(t,\omega)}_T(\cdot,\omega)$ for $\bP$-almost all $\omega\in\Omega$. Thus, $\bP$-almost surely
\begin{align*}
D^{e_i}X^{s,\xi}|_{[t,T]}&=C([t,T],\bR^d)\text{-}\lim_{n\to\infty}\frac{X^{s,\xi+h_ne_i}|_{[t,T]}-X^{s,\xi}|_{[t,T]}}{h_n}\\
&=C([t,T],\bR^d)\text{-}\lim_{n\to\infty}\frac{X^{t,X^{s,\xi+h_ne_i}(t)}_T-X^{t,X^{s,\xi}(t)}_T}{h_n}
=D X^{t,X^{s,\xi}(t)}_T D^{e_i} X^{s,\xi}(t),
\end{align*}
by the chain rule and using Theorem~\ref{thm:regularityInitialCond}(i).
\item
The general assertion follows by induction over $|\alpha|$, using similar arguments as in the proof of part (i). \qedhere
\end{enumerate}
\end{proof}

\section{Regularity of the functional $F^\epsilon$}
\label{sec:RegularityFFepsilon}

Recall the definition \eqref{eq:defFteps} of the mappings $F_t^\epsilon$ from $D([0,t],\bR^d)$ to $\bR$, $t\in[0,T]$, in Section~\ref{sec:Introduction}, i.e.,
\begin{equation*}
F_t^\epsilon(x):=\bE f^\epsilon(X^{t,x}_T),\qquad x\in D([0,t],\bR^d),
\end{equation*}
where $\epsilon>0$ and $f^\epsilon=f\circ M^\epsilon\colon D([0,T],\bR^d)\to\bR$ is the regularized version of \linebreak$f\colon C([0,T],\bR^d)\to\bR$ defined by \eqref{eq:deffepsilon}, \eqref{eq:defMepsilon1}, \eqref{eq:defMepsilon2}.

Our minimal assumption on $f$ is that it is $\cB(C([0,T],\bR^d))/\cB(\bR)$-measurable and has polynomial growth.
Obviously, under this assumption, $F_t^\epsilon=(F_t^\epsilon)_{t\in[0,T]}$ is a non-anticipative functional on $D([0,T],\bR^d)$ in the sense of Definition~\ref{def:nonanticipativeFunctional}.
The goal of this section is to show that, if  $f\in C^2_p(C([0,T],\bR^d),\bR)$, then $F^\varepsilon$ is a regular functional belonging the class $\bC^{1,2}_b([0,T])$ introduced in Definition~\ref{def:regularFunctionals}.
We divide the proof into a series of lemmata. In the proofs we often use the fact that if for some $n\in\bN_0$ the polynomial growth bound
\begin{equation}\label{eq:dnf}
\|D^{n}f(x)\|_{\bo^{(n)}(C([0,T],\bR^d),\bR)}\le C\left(1+\|x\|^q_{C([0,T],\bR^d)}\right),~x\in C([0,T],\bR^d),
\end{equation}
holds, then
$$
\|D^{n}f^{\epsilon}(x)\|_{\bo^{(n)}(D([0,T],\bR^d),\bR)}\le C\left(1+\|x\|^q_{D([0,T],\bR^d)}\right),~x\in D([0,T],\bR^d),
$$
with the same $C$ as in \eqref{eq:dnf} independently of $\epsilon$. This is the consequence of the chain rule
and the equality $\|M^\epsilon\|_{\bo(D([0,T],\bR^d),C([0,T],\bR^d))}=1$.
%

\begin{lemma}\label{lem:regularityFepsilon1}
For $f\in C_p(C([0,T],\bR^d),\bR)$ and $\epsilon>0$, the non-anticipative functional $F^\epsilon=(F^\epsilon_t)_{t\in[0,T]}$ defined by \eqref{eq:defFteps} is left-continuous and boundedness-preserving, i.e., $F^\epsilon\in\bC^{0,0}_l([0,T])\cap\bB([0,T])$.
Moreover,
\[|F^\epsilon_t(x)|\leq C \big(1+\nnrm{x}{D([0,t],\bR^d)}^q\big)\]
for all $t\in[0,T]$ and $x\in D([0,t],\bR^d)$, where $C,q\in(0,\infty)$ do not depend on $t$, $x$ or $\epsilon$.
\end{lemma}

\begin{proof}
In order to verify the left-continuity, it suffices to show the following: For every $x=x_t\in D([0,t],\bR^d)\subset\Lambda$ and every sequence $(x^n)_{n\in\bN}\subset\Lambda$ with $x^n=x^n_{t^n}\in D([0,t^n],\bR^d)$, $t^n\in[0,t]$, and
$d_\infty(x^n,x)\xrightarrow{n\to\infty}0$, we have
$F^\epsilon_{t^n}(x^n)\xrightarrow{n\to\infty} F^\epsilon_t(x)$. Applying Lemma~\ref{lem:LpConvImpliesWeakConv} with $B=D([0,T],\bR^d)$, $S=\bR$, $Y=X^{t,x}_T$, $Y_n=X^{t^n,x^n}_T$ and $\varphi=f^\epsilon$, it is enough to prove that
\begin{equation}\label{eq:proofFepsCl1}
X^{t^n,x^n}_T\xrightarrow{n\to\infty} X^{t,x}_T\quad\text{ in }L^p(\Omega;D([0,T],\bR^d))
\end{equation}
for every $p\geq1$. To this end, we start by estimating
\begin{equation}\label{eq:proofFepsCl2}
\begin{aligned}
&\gnnrm{X^{t,x_t}_T-X^{t^n,x^n_{t^n}}_T}{D([0,T],\bR^d)}\\
&\qquad\leq \gnnrm{X^{t,x_t}_T-X^{t,x^n_{t^n,t-t^n}}_T}{D([0,T],\bR^d)} + \gnnrm{X^{t,x^n_{t^n,t-t^n}}_T-X^{t^n,x^n_{t^n}}_T}{D([0,T],\bR^d)}\\
&\qquad=:A+B
\end{aligned}
\end{equation}
and deal with each term separately. Concerning the first term, note that
\begin{equation}\label{eq:proofFepsCl3}
\begin{aligned}
\bE(A^p)&\leq 2^{p-1}\Big(d_\infty(x,x^n)^p+\bE\big(\sup_{s\in[t,T]}\big|X^{t,x(t)}(s)-X^{t,x^n(t^n)}(s)\big|^p\big)\Big)\\
&\leq C \, d_\infty(x,x^n)^p,
\end{aligned}
\end{equation}
where the second estimate follows from \eqref{eq:standardEstInitialCond2} and the definition of the metric $d_\infty$. Since
\[X^{t^n,x^n(t^n)}|_{[t,T]}=X^{t,X^{t^n,x^n(t^n)}(t)}_T\quad \text{ $\bP$-almost surely}\]
as an equality in $C([t,T],\bR^d)$, the $p$-th moment of the second term in \eqref{eq:proofFepsCl2} is bounded by
\begin{equation*}
\begin{aligned}
\bE(B^p)&\leq 2^{p-1}\Big(
\bE\big(\sup_{s\in[t^n,t]}\big|x^n(t^n)-X^{t^n,x^n(t^n)}(s)\big|^p\big)\\
&\quad + \bE\big(\sup_{s\in[t,T]}\big|X^{t,x^n(t^n)}(s)-X^{t,X^{t^n,x^n(t^n)}(t)}(s)\big|^p\big)\Big)\\
&\leq C\,\bE\big(\sup_{s\in[t^n,t]}\big|x^n(t^n)-X^{t^n,x^n(t^n)}(s)\big|^p\big),
\end{aligned}
\end{equation*}
where we used again the estimate \eqref{eq:standardEstInitialCond2} in the second step. Taking into account the time-homogeneity of Eq.~\eqref{eq:SDE} and using the estimate \eqref{eq:standardEstInitialCond2} once more, we obtain
\begin{equation}\label{eq:proofFepsCl4}
\begin{aligned}
\bE(B^p)&\leq C\Big(|x^n(t^n)-x(t)|^p+\bE\big(\sup_{s\in[t^n,t]}\big|x(t)-X^{t^n,x(t)}(s)\big|^p\big)\\
&\quad+\bE\big(\sup_{s\in[t^n,t]}\big|X^{t^n,x(t)}(s)-X^{t^n,x^n(t^n)}(s)\big|^p\big)\Big)\\
&\leq C\Big(|x^n(t^n)-x(t)|^p+\bE\big(\sup_{s\in[0,t-t^n]}\big|x(t)-X^{0,x(t)}(s)\big|^p\big)+|x(t)-x^n(t^n)|^p\Big)\\
&\leq C\Big(d_\infty(x,x^n)^p+ \bE\big(\sup_{s\in[0,t-t^n]}\big|x(t)-X^{0,x(t)}(s)\big|^p\big)\Big).
\end{aligned}
\end{equation}
By dominated convergence, the expectation in the last line goes to zero as $n\to\infty$. The combination of \eqref{eq:proofFepsCl2}, \eqref{eq:proofFepsCl3} and \eqref{eq:proofFepsCl4} yields \eqref{eq:proofFepsCl1} and thus the left-continuity of $F^\epsilon$.

To see that $F^\epsilon$ is boundedness-preserving, we use the polynomial growth of \linebreak $f^\epsilon\colon D([0,T],\bR^d)\to\bR$ and estimate \eqref{eq:standardEstInitialCond1} to conclude that, for all $t\in[0,T]$ and $x\in D([0,t],\bR^d)$,
\begin{align*}
|F^\epsilon_t(x)|=|\bE f^\epsilon(X^{t,x}_T)|&\leq \bE\, C\big(1+\nnrm{X^{t,x}_T}{D([0,T],\bR^d)}^q\big)\\
&\leq C\big(1+\nnrm{x}{D([0,t],\bR^d)}^q+\bE\big(\nnrm{X^{t,x(t)}_T}{D([t,T],\bR^d)}^q\big)\big)\\
&\leq C\big(1+\nnrm{x}{D([0,t],\bR^d)}^q\big)
\end{align*}
where the exponent $q\in(1,\infty)$ and the constant $C\in(0,\infty)$ do not depend on $t$, $x$ or $\epsilon$.
\end{proof}

\begin{lemma}\label{lem:regularityFepsilon2}
If $f\in C^1_p(C([0,T],\bR^d),\bR)$ and $\epsilon>0$, the non-anticipative functional $F^\epsilon=(F^\epsilon_t)_{t\in[0,T]}$ defined by \eqref{eq:defFteps} is vertically differentiable. The vertical derivative $\nabla_x F^\epsilon=(\nabla_x F^\epsilon_t)_{t\in[0,T]}$ is left-continuous and boundedness-preserving, i.e., $\nabla_x F^\epsilon\in\bC^{0,0}_l([0,T])\cap\bB([0,T])$, and is given by
\begin{equation}\label{eq:nablaF}
\nabla_x F^\epsilon_t(x)=\Big(\bE\big[Df^\epsilon(X^{t,x}_T)\,(\one_{[t,T]}D^{e_1}\!X^{t,x(t)}_T)\big],\ldots,\bE\big[Df^\epsilon(X^{t,x}_T)\,(\one_{[t,T]}D^{e_d}\!X^{t,x(t)}_T)\big]\Big)\in\bR^d,
\end{equation}
$t\in [0,T]$, $x\in D([0,t],\bR^d)$. Moreover,
\[|\nabla_x F^\epsilon_t(x)|\leq C \big(1+\nnrm{x}{D([0,t],\bR^d)}^q\big)\]
for all $t\in[0,T]$ and $x\in D([0,t],\bR^d)$, where $C,\,q\in(0,\infty)$ do not depend on $t$, $x$ or $\epsilon$.
\end{lemma}

\begin{proof}
To show the vertical differentiability, we fix $t\in[0,T]$, $x=x_t\in D([0,t],\bR^d)$, $i\in\{1,\ldots,d\}$ and apply the differentiation lemma for parameter-dependent integrals to the mapping
\[(-\delta,\delta)\times\Omega\ni(h,\omega)\mapsto f^\epsilon\big(X^{t,x_t^{he_i}}_T(\omega)\big)\in\bR,\]
where $\delta>0$ and $x_t^{he_i}\in D([0,t],\bR^d)$ is the vertical perturbation of $x_t$ by $he_i\in\bR^d$. The polynomial growth of $Df\colon C([0,T],\bR^d)\to\bo(C([0,T],\bR^d),\bR)$ implies polynomial growth of $Df^\epsilon\colon D([0,T],\bR^d)\to\bo(D([0,T],\bR^d),\bR)$. Together with Theorem~\ref{thm:regularityInitialCond}(i) this implies that
 there exist $C,q\in(0,\infty)$ such that, for all $h\in(-\delta,\delta)$,
\begin{align*}
\big|\frac{\dl}{\dl h}&f^\epsilon\big(X^{t,x_t^{he_i}}_T\big)\big|
=\big|Df^\epsilon\big(X^{t,x_t^{he_i}}_T\big)\,(\one_{[t,T]}D^{e_i}X^{t,x(t)+he_i}_T)\big|\\
&\leq \gnnrm{Df^\epsilon\big(X^{t,x_t^{he_i}}_T\big)}{\bo(D([0,T],\bR^d),\bR)}\gnnrm{\one_{[t,T]}D^{e_i}X^{t,x(t)+he_i}_T}{D([0,T],\bR^d)}\\
&\leq C\big(1+\nnrm{X^{t,x_t^{he_i}}_T}{D([0,T],\bR^d)}\big)^q\,\nnrm{D^{e_i}X^{t,x(t)+he_i}_T}{C([t,T],\bR^d)}\\
&\leq C\big(1+\nnrm{x_t}{D([0,t],\bR^d)}+\sup_{\xi\in B_\delta(x(t))}\nnrm{X^{t,\xi}_T}{C([t,T],\bR^d)}\big)^q\,\sup_{\xi\in B_\delta(x(t))}\nnrm{D^{e_i}X^{t,\xi}_T}{C([t,T],\bR^d)},
\end{align*}
where the last the upper bound belongs to $L^p(\Omega)$ for every $p\in[1,\infty)$ due to \eqref{eq:unifLpEstDalphaX}. Thus, we can apply the differentiation lemma for parameter-dependent integrals and use the chain rule together with Theorem~\ref{thm:regularityInitialCond}(i) to obtain
\begin{align*}
\frac{\dl}{\dl h}\bE\big[f^\epsilon\big(X^{t,x_t^{he_i}}_T\big)\big]\big|_{h=0}
&=\bE\Big[\frac{\dl}{\dl h}f^\epsilon\big(X^{t,x_t^{he_i}}_T\big)\big|_{h=0}\Big]
=\bE\big[Df^\epsilon(X^{t,x_t}_T)\,(\one_{[t,T]}D^{e_i}\!X^{t,x(t)}_T)\big].
\end{align*}

Next, we verify the left-continuity of $\nabla_xF^\epsilon$. To this end, it suffices to prove the following assertion: For every $x=x_t\in D([0,t],\bR^d)\subset\Lambda$ and every sequence $(x^n)_{n\in\bN}\subset\Lambda$ with $x^n=x^n_{t^n}\in D([0,t^n],\bR^d)$, $t^n\in[0,t]$, and
$d_\infty(x^n,x)\xrightarrow{n\to\infty}0$, there exists a subsequence $(x^{n_k})_{k\in\bN}$ such that
$\nabla_x F^\epsilon_{t^{n_k}}(x^{n_k})\xrightarrow{k\to\infty} \nabla_x F^\epsilon_t(x)$. Fix $x=x_t$ and such a sequence $(x^n)_{n\in\bN}\subset\Lambda$. For $i\in\{1,\ldots,d\}$,
\begin{equation}\label{eq:proofNablaFepsCl1}
\begin{aligned}
\big|\bE\big[Df^\epsilon(X^{t,x}_T)\,&(\one_{[t,T]}D^{e_i}\!X^{t,x(t)}_T)\big]-\bE\big[Df^\epsilon(X^{t^n,x^n}_T)\,(\one_{[t^n,T]}D^{e_i}\!X^{t^n,x^n(t^n)}_T)\big]\big|\\
&\leq \big|\bE\big[\big(Df^\epsilon(X^{t,x}_T)-Df^\epsilon(X^{t^n,x^n}_T)\big)\,(\one_{[t,T]}D^{e_i}\!X^{t,x(t)}_T)\big]\big|\\
&\quad + \big|\bE\big[Df^\epsilon(X^{t^n,x^n}_T)\,\big(\one_{[t,T]}D^{e_i}\!X^{t,x(t)}_T-\one_{[t^n,T]}D^{e_i}\!X^{t^n,x^n(t^n)}_T\big)\big]\big|\\
&\quad =: A+B.
\end{aligned}
\end{equation}
By the convergence \eqref{eq:proofFepsCl1}, by Lemma~\ref{lem:LpConvImpliesWeakConv} with $B=D([0,T],\bR^d)$, $S=\bo(D([0,T],\bR^d),\bR)$, $Y=X^{t,x}_T$, $Y_n=X^{t^n,x^n}_T$ and $\varphi=D f^\epsilon$,
and by the estimate \eqref{eq:unifLpEstDalphaX}, the first term in \eqref{eq:proofNablaFepsCl1} satisfies
\begin{equation}\label{eq:proofNablaFepsCl2}
A\leq \big(\bE\gnnrm{Df^\epsilon(X^{t,x}_T)-Df^\epsilon(X^{t^n,x^n}_T)}{\bo(D([0,T],\bR^d),\bR)}^2\big)^{\frac12}\big(\bE\nnrm{D^{e_i}\!X^{t,x(t)}_T}{C([t,T],\bR^d)}^2\big)^{\frac12}\xrightarrow{n\to\infty}0.
\end{equation}
The second term in \eqref{eq:proofNablaFepsCl1} can be estimated by
\begin{equation}\label{eq:proofNablaFepsCl3}
\begin{aligned}
B&\leq
\Big(\bE\gnnrm{Df(M^\epsilon X^{t^n,x^n}_T)M_\epsilon}{\bo(L^1([0,T],\bR^d),\bR)}^2\Big)^{\frac12}\\
&\quad\times
\Big(\bE\gnnrm{\one_{[t,T]}D^{e_i}\!X^{t,x(t)}_T-\one_{[t^n,T]}D^{e_i}\!X^{t^n,x^n(t^n)}_T}{L^1([0,T],\bR^d)}^2\Big)^{\frac12}\\
&\leq \|M_\varepsilon\|_{\bo(C([0,T],\bR^d),L^1([0,T],\bR^d)}\Big(\bE\gnnrm{Df(M^\epsilon X^{t^n,x^n}_T)}{\bo(C([0,T],\bR^d),\bR)}^2\Big)^{\frac12}\\
&\quad\times
\Big(\bE\gnnrm{\one_{[t,T]}D^{e_i}\!X^{t,x(t)}_T-\one_{[t^n,T]}D^{e_i}\!X^{t^n,x^n(t^n)}_T}{L^1([0,T],\bR^d)}^2\Big)^{\frac12}\\
&=:B_{\epsilon}\,B_1\,B_2,
\end{aligned}
\end{equation}
where $B_1$ bounded uniformly in $n\in\bN$ due to the polynomial growth of $$Df\colon C([0,T],\bR^d)\to\bo(C([0,T],\bR^d),\bR),$$ the estimate~\eqref{eq:unifLpEstDalphaX}, and since $|x(t)-x^n(t^n)|\leq d_\infty(x,x^n)\xrightarrow{n\to\infty}0$.
  Note also that $B_{\epsilon}=\|M^\varepsilon\|_{\bo(C([0,T],\bR^d),L^1([0,T],\bR^d)}=\sup_{s\in\bR}|\eta_\epsilon(s)|=(\epsilon/2)^{-1}$. We further have
\begin{equation}\label{eq:proofNablaFepsCl4}
\begin{aligned}
B_2&
\leq \Big(\bE\gnnrm{\one_{[t^n,t]}D^{e_i}\!X^{t^n,x^n(t^n)}_T}{L^1([0,T];\bR^d)}^2\Big)^{\frac12}\\
&\quad+C_T\Big(\bE\gnnrm{D^{e_i}\!X^{t,x(t)}_T-(D^{e_i}\!X^{t^n,x^n(t^n)})|_{[t,T]}}{C([t,T],\bR^d)}^2\Big)^{\frac12}\\
&=:B_{21}+C_TB_{22}.
\end{aligned}
\end{equation}
Using the time-homogeneity of Eq.~\eqref{eq:SDE} and the estimate~\eqref{eq:unifLpEstDalphaX}, one sees that the term $B_{21}$ in \eqref{eq:proofNablaFepsCl4} tends to zero as $n\to\infty$ since
\begin{equation}\label{eq:proofNablaFepsCl5}
\begin{aligned}
\gnnrm{\one_{[t^n,t]}D^{e_i}\!X^{t^n,x^n(t^n)}_T}{L^1([0,T];\bR^d)}&\sim\gnnrm{D^{e_i}\!X^{0,x^n(t^n)}_{t-t^n}}{L^1([0,t-t^n];\bR^d)}\\
&\leq (t-t^n)\sup_{\xi\in B_1(x(t))}\nnrm{D^{e_i}\!X^{0,\xi}_T}{C([0,T],\bR^d)}.
\end{aligned}
\end{equation}
for $n$ large enough. Concerning the term $B_{22}$ in \eqref{eq:proofNablaFepsCl4} note that
\begin{equation}\label{eq:proofNablaFepsCl6a}
\begin{aligned}
\gnnrm{D^{e_i}\!X^{t,x(t)}_T-(D^{e_i}\!X^{t^n,x^n(t^n)})|_{[t,T]}}{C([t,T],\bR^d)}
&\leq \gnnrm{D^{e_i}\!X^{t,x(t)}_T-D^{e_i}\!X^{t,x^n(t^n)}_T}{C([t,T],\bR^d)}\\
&\quad +\gnnrm{D^{e_i}\!X^{t,x^n(t^n)}_T-(D^{e_i}\!X^{t^n,x^n(t^n)})|_{[t,T]}}{C([t,T],\bR^d)},
\end{aligned}
\end{equation}
where the $L^2(\bP)$-norm of the first term on the right hand side goes to zero as $n\to\infty$ due to Corollary~\ref{cor:ContDalphaX}. For the second term on the right hand side of \eqref{eq:proofNablaFepsCl6a} we use Remark~\ref{rem:FrechetDerXsxi} and Corollary~\ref{cor:chainRuleDalphaX} to obtain
\begin{equation}\label{eq:proofNablaFepsCl6}
\begin{aligned}
\big\|D^{e_i}\!X^{t,x^n(t^n)}_T&-(D^{e_i}\!X^{t^n,x^n(t^n)})|_{[t,T]}\big\|_{C([t,T],\bR^d)}\\
&=\gnnrm{D\!X^{t,x^n(t^n)}_T\,e_i-D X^{t,X^{t^n,x^n(t^n)}(t)}_T\,D^{e_i}\!X^{t^n,x^n(t^n)}(t)}{C([t,T],\bR^d)}\\
&\leq\gnnrm{\big(D\!X^{t,x^n(t^n)}_T-D X^{t,X^{t^n,x^n(t^n)}(t)}_T\big)\,e_i}{C([t,T],\bR^d)}\\
&\quad+\gnnrm{D X^{t,X^{t^n,x^n(t^n)}(t)}_T\,\big(e_i-D^{e_i}\!X^{t^n,x^n(t^n)}(t)\big)}{C([t,T],\bR^d)}.
\end{aligned}
\end{equation}
Applying Corollary~\ref{cor:ContDalphaX}, arguing as in \eqref{eq:proofFepsCl4}, and using the fact that $L^p(\bP)$-convergence implies almost-sure convergence for a subsequence, one sees that
\begin{equation*}
\gnnrm{\big(D\!X^{t,x^{n_k}(t^{n_k})}_T-D X^{t,X^{t^{n_k},x^{n_k}(t^{n_k})}(t)}_T\big)\,e_i}{C([t,T],\bR^d)}\xrightarrow{k\to\infty}0
\end{equation*}
in $L^2(\bP)$ for an increasing sequence $(n_k)_{k\in\bN}\subset\bN$. Finally, the second term on the right hand side of \eqref{eq:proofNablaFepsCl6} tends to zero as $n\to\infty$ by Theorem~\ref{thm:regularityInitialCond}(ii) and a dominated convergence argument.
Thus, in summary, the estimates \eqref{eq:proofNablaFepsCl1}---\eqref{eq:proofNablaFepsCl6} yield the left-continuity of $\nabla_x F^\epsilon$.

To see that $\nabla_x F^\epsilon$ is boundedness-preserving, we use the polynomial growth of \linebreak$Df^\epsilon\colon D([0,T],\bR^d)\to\bo(D([0,T],\bR^d),\bR)$, Theorem~\ref{thm:regularityInitialCond}(ii) and the estimate \eqref{eq:standardEstInitialCond1} to conclude that  for all $t\in[0,T]$ and $x\in D([0,t],\bR^d)$,
\begin{equation}\label{eq:dfpg}
\begin{aligned}
\big|\bE\big[&Df^\epsilon(X^{t,x_t}_T)\,(\one_{[t,T]}D^{e_i}\!X^{t,x(t)}_T)\big]\big|\\
&\leq \big(\bE\gnnrm{Df^\epsilon\big(X^{t,x_t}_T\big)}{\bo(D([0,T],\bR^d),\bR)}^2\big)^{\frac12}\big(\bE\gnnrm{\one_{[t,T]}D^{e_i}X^{t,x(t)}_T}{D([0,T],\bR^d)}^2\big)^{\frac12}\\
&\leq C\,\big(\bE(1+\nnrm{X^{t,x_t}_T}{D([0,T],\bR^d)}^p)^2\big)^{\frac12}\sup_{\xi\in\bR^d}\big(\bE\nnrm{D^{e_i}X^{t,\xi}_T}{C([t,T],\bR^d)}^2\big)^{\frac12}\\
&\leq C\,\big(\bE (1+\nnrm{X^{t,x_t}_T}{D([0,T],\bR^d)}^{2p})\big)^{\frac12}\\
&\leq C\,\big(1+\nnrm{x_t}{D([0,t],\bR^d)}^q\big)
\end{aligned}
\end{equation}
where the exponents $p,\,q\in[1,\infty)$ and the constants $C\in(0,\infty)$ are suitably chosen and do not depend on $t$, $x$ or $\epsilon$.
\end{proof}

\begin{lemma}\label{lem:regularityFepsilon3}
If $f\in C^2_p(C([0,T],\bR^d),\bR)$ and $\epsilon>0$, the non-anticipative functional $F^\epsilon=(F^\epsilon_t)_{t\in[0,T]}$ defined by \eqref{eq:defFteps} is twice vertically differentiable. The second vertical derivative $\nabla_x^2 F^\epsilon=(\nabla_x^2 F^\epsilon_t)_{t\in[0,T]}$ is left-continuous and boundedness-preserving, i.e., $\nabla_x^2 F^\epsilon\in\bC^{0,0}_l([0,T])\cap\bB([0,T])$, and is given by
\begin{equation}\label{eq:nanab}
\begin{aligned}
&(\nabla_x(\nabla_x F^{\epsilon}_t)_i)_j=(\nabla_x^2 F^\epsilon_t(x))(e_i,e_j)\\
&=\bE\big[D^2f^\epsilon(X^{t,x}_T)\,\big(\one_{[t,T]}D^{e_i}X^{t,x(t)}_T,\,\one_{[t,T]}D^{e_j}X^{t,x(t)}_T\big) +Df^\epsilon(X^{t,x})\,\big(\one_{[t,T]}D^{e_i+e_j}X^{t,x(t)}_T\big)\big]
\end{aligned}
\end{equation}
$t\in [0,T]$, $x\in D([0,t],\bR^d)$, $i,j\in\{1,\ldots,d\}$. Moreover,
\begin{equation}\label{eq:nab2e}
|\nabla_x^2 F^\epsilon_t(x)|\leq C \big(1+\nnrm{x}{D([0,t],\bR^d)}^q\big)
\end{equation}
for all $t\in[0,T]$ and $x\in D([0,t],\bR^d)$, where $C,\,q\in(0,\infty)$ does not depend on $t$, $x$ or $\epsilon$.
\end{lemma}

\begin{proof}
The proof of the statement follows a line analogous to the proof of Lemma \ref{lem:regularityFepsilon2} and therefore we only give a short sketch. We fix $t\in[0,T]$, $x=x_t\in D([0,t],\bR^d)$, $i,j\in\{1,\ldots,d\}$ and apply the differentiation lemma for parameter-dependent integrals to the mapping
\[(-\delta,\delta)\times\Omega\ni(h,\omega)\mapsto Df^\epsilon\big(X^{t,x_t^{he_j}}_T(\omega)\big)\big(\one_{[t,T]}D^{e_i}\!X^{t,x(t)+he_j}_T(\omega)\big)\in\bR,\]
where $\delta>0$ and $x_t^{he_j}\in D([0,t],\bR^d)$ is the vertical perturbation of $x_t$ by $he_j\in\bR^d$. For fixed $\omega$, we apply the product rule to a mapping of the form $(-\delta,\delta)\ni h\mapsto A_hf_h$, where $h\mapsto A_h\in \bo(D([0,T],\bR^d),\bR)$ and $h\mapsto f_h\in D([0,T],\bR^d)$ are Fr\'echet differentiable, which takes the usual form (with an analogous proof to the real case) $\frac{\dl}{dh}(A_hf_h)=(\frac{\dl}{\dl h}A_h)f_h+A_h\frac{\dl}{\dl h}f_h$. Furthermore, $A_h$ takes the form $A_h=B(g_h)$ where $h\mapsto g_h\in D([0,T],\bR^d)$ and $B\colon D([0,T],\bR^d)\to \bo(D([0,T],\bR^d),\bR)$ are also Fr\'echet differentiable and hence, by the chain rule,  $\frac{\dl}{dh}(A_hf_h)=DB(g_h)[\frac{\dl}{dh}g_h]f_h+A_h\frac{\dl}{\dl h}f_h$. Thus,
\begin{equation}
\begin{aligned}
&\frac{\dl}{\dl h}Df^\epsilon\big(X^{t,x_t^{he_j}}_T(\omega)\big)\big(\one_{[t,T]}D^{e_i}\!X^{t,x(t)+he_j}_T(\omega)\big)\\
&=D^2f^\epsilon\big(X^{t,x_t^{he_j}}_T(\omega)\big)\big(\one_{[t,T]}D^{e_j}\!X^{t,x(t)+he_j}_T(\omega),\one_{[t,T]}D^{e_i}\!X^{t,x(t)+he_j}_T(\omega)\big)\\
&\quad\quad+Df^\epsilon\big(X^{t,x_t^{he_j}}_T(\omega)\big)\big(\one_{[t,T]}D^{e_i}D^{e_j}\!X^{t,x(t)+he_j}_T(\omega)\big).
\end{aligned}
\end{equation}
Using the polynomial growth of $Df^\epsilon$ and $D^2f^\epsilon$ together with Theorem~\ref{thm:regularityInitialCond}(i) this implies, as in the proof of Lemma \ref{lem:regularityFepsilon2}, that there exist $C,q\in(0,\infty)$ such that, for all $h\in(-\delta,\delta)$
\begin{equation*}
\begin{aligned}
&\left|\frac{\dl}{\dl h}Df^\epsilon\big(X^{t,x_t^{he_j}}_T\big)\big(\one_{[t,T]}D^{e_i}\!X^{t,x(t)+he_j}_T\big)\right|\\
&\le C\left(1+\nnrm{x_t}{D([0,t],\bR^d)}+\sup_{\xi\in B_\delta(x(t))}\nnrm{X^{t,\xi}_T}{C([t,T],\bR^d)}\right)^q\times\\
&\quad\left(\sup_{\xi\in B_\delta(x(t))}\nnrm{D^{e_i}X^{t,\xi}_T}{C([t,T],\bR^d)}\sup_{\xi\in B_\delta(x(t))}\nnrm{D^{e_j}X^{t,\xi}_T}{C([t,T],\bR^d)}\right.\\
&\qquad
+\left.
\sup_{\xi\in B_\delta(x(t))}\nnrm{D^{e_i+e_j}X^{t,\xi}_T}{C([t,T],\bR^d)}\right),
\end{aligned}
\end{equation*}
where the last the upper bound belongs to $L^p(\Omega)$ for every $p\in[1,\infty)$ by \eqref{eq:unifLpEstDalphaX}. Therefore, using also the symmetry of $D^2f^\epsilon$,
\begin{equation*}
\begin{aligned}
&(\nabla_x(\nabla_x F^{\epsilon}_t)_i)_j=\frac{\dl}{\dl h}\Big(\bE\Big[ Df^\epsilon\big(X^{t,x_t^{he_j}}_T\big)\big(\one_{[t,T]}D^{e_i}\!X^{t,x(t)+he_j}_T\big)\Big]\Big)\Big|_{h=0}\\
&=\bE\Big[\frac{\dl}{\dl h}\big(Df^\epsilon\big(X^{t,x_t^{he_j}}_T\big)\big(\one_{[t,T]}D^{e_i}\!X^{t,x(t)+he_j}_T\big)\big)\big|_{h=0}\Big]\\
&=\bE\big[D^2f^\epsilon(X^{t,x_t}_T)\,\big(\one_{[t,T]}D^{e_i}X^{t,x(t)}_T,\,\one_{[t,T]}D^{e_j}X^{t,x(t)}_T\big)  +Df^\epsilon(X^{t,x_t}_T)\,\big(\one_{[t,T]}D^{e_i+e_j}X^{t,x(t)}_T\big)\big].
\end{aligned}
\end{equation*}
The proof of the left continuity of the second term is essentially identical to the proof of the left continuity of $\nabla_x F^{\epsilon}$ .
For the left continuity of the first term one uses a telescoping sum and H\"older's inequality to get
\begin{equation*}
\begin{aligned}
&\left|\bE\big[D^2f^\epsilon(X^{t,x})\,\big(\one_{[t,T]}D^{e_i}X^{t,x(t)}_T,\,\one_{[t,T]}D^{e_j}X^{t,x(t)}_T\big)\big]\right. \\
&\left.-\bE\big[D^2f^\epsilon(X^{t^n,x^n})\,\big(\one_{[t^n,T]}D^{e_i}X^{t^n,x^n(t)}_T,\,\one_{[t^n,T]}D^{e_j}X^{t^n,x^n(t)}_T\big)\big]\right|
:=|\bE(A(u,v)-A_n(u_n,v_n))|\\
&\le |\bE(A_n(u_n,v-v_n))|+|\bE((A-A_n)(u_n,v))|+|\bE(A(u-u_n,v))|\\
&\le \|A_n\|_{L^4(\Omega;\bo^{(2)}(L^1([0,T],\bR^d),\bR))}\|u_n\|_{L^4(\Omega;L^1([0,T],\bR^d))}\|v-v_n\|_{L^2(\Omega,L^1([0,T],\bR^d))}\\
&+\|A-A_n\|_{L^2(\Omega;\bo^{(2)}(D([0,T],\bR^d),\bR))}\|u_n\|_{L^4(\Omega;D([0,T],\bR^d))}\|v\|_{L^4(\Omega,D([0,T],\bR^d))}\\
&+\|A\|_{L^4(\Omega;\bo^{(2)}(L^1([0,T],\bR^d),\bR))}\|u-u_n\|_{L^2(\Omega;L^1([0,T],\bR^d))}\|v_n\|_{L^4(\Omega;L^1([0,T],\bR^d))}\\
&:=a_n+b_n+c_n.
\end{aligned}
\end{equation*}
Now $b_n$ can be treated as \eqref{eq:proofNablaFepsCl2} using Lemma~\ref{lem:LpConvImpliesWeakConv} with
$$
B=D([0,T],\bR^d),~S=\bo^{(2)}(D([0,T],\bR^d),\bR),~Y=X^{t,x}_T,~Y_n=X^{t^n,x^n}_T,~\varphi=D^2 f^\epsilon.
$$
The terms $a_n$ and $c_n$ can be handled analogously to error term $B$ in \eqref{eq:proofNablaFepsCl3}, where we first select a subsequence such that $a_{n_k}\to 0$, then a further subsequence such that $c_{n_{k_l}}\to 0$. This will finally show that for every $x=x_t\in D([0,t],\bR^d)\subset\Lambda$ and every sequence $(x^n)_{n\in\bN}\subset\Lambda$ with $x^n=x^n_{t^n}\in D([0,t^n],\bR^d)$, $t^n\in[0,t]$, and
$d_\infty(x^n,x)\xrightarrow{n\to\infty}0$, there exists a subsequence $(x^{n_{k_l}})_{l\in\bN}$ such that
$\nabla^2_x F^\epsilon_{t^{n_{k_l}}}(x^{n_{k_l}})\xrightarrow{l\to\infty} \nabla^2_xF^\epsilon_t(x)$ verifying the left-continuity of $\nabla^2_x F$.

Finally, the estimate \eqref{eq:nab2e} (and hence that $\nabla^2_xF$ is boundedness preserving) follows from \eqref{eq:nanab} by analogous estimates as in \eqref{eq:dfpg}, using the polynomial growth of $$Df^\epsilon\colon D([0,T],\bR^d)\to\bo(D([0,T],\bR^d),\bR)$$ and $D^2f^\epsilon\colon D([0,T],\bR^d)\to\bo^{(2)}(D([0,T],\bR^d),\bR)$ combined with Theorem~\ref{thm:regularityInitialCond}(ii) and the estimate \eqref{eq:standardEstInitialCond1}.
\end{proof}
\begin{remark}\label{rem:nabn}
In a completely analogous fashion, with more notational effort, one can prove that
 if $f\in C^n_p(C([0,T],\bR^d),\bR)$ and $\epsilon>0$, then the non-anticipative functional $F^\epsilon=(F^\epsilon_t)_{t\in[0,T]}$ defined by \eqref{eq:defFteps} is $n$-times vertically differentiable, $n\in\bN$. The $n$-th vertical derivative $\nabla_x^n F^\epsilon=(\nabla_x^n F^\epsilon_t)_{t\in[0,T]}$ is left-continuous and boundedness-preserving, i.e., $\nabla_x^n F^\epsilon\in\bC^{0,0}_l([0,T])\cap\bB([0,T])$, and
\begin{equation*}
|\nabla_x^n F^\epsilon_t(x)|\leq C \big(1+\nnrm{x}{D([0,t],\bR^d)}^q\big)
\end{equation*}
for all $t\in[0,T]$ and $x\in D([0,t],\bR^d)$, where $C,\,q\in(0,\infty)$ does not depend on $t$, $x$ or $\epsilon$.
\end{remark}

\begin{lemma}\label{lem:regularityFepsilon4}
If $f\in C^1_p(C([0,T],\bR^d),\bR)$ and $\epsilon>0$, the non-anticipative functional $F^\epsilon=(F^\epsilon_t)_{t\in[0,T]}$ defined by \eqref{eq:defFteps} is horizontally differentiable. The horizontal derivative $\cD F^\epsilon=(\cD_t F^\epsilon)_{t\in[0,T)}$ is continuous at fixed times, and the extension $(\cD_t F^\epsilon)_{t\in[0,T]}$ of $(\cD_t F^\epsilon)_{t\in[0,T)}$ by zero belongs to the class $\bB([0,T])$. The horizontal derivative is given by
\begin{equation}\label{eq:dtf}
\cD_t F^\epsilon(x)=\bE \big[ Df(M^\epsilon X^{t,x}_{T})\,\big(x(t)\eta_\epsilon(\cdot -t)-\int_{t}^{T}\eta_{\epsilon}'(\cdot-r)X^{t,x(t)}(r)\,\dl r\big)\big],
\end{equation}
$t\in [0,T)$, $x\in D([0,t],\bR^d)$. Moreover,
\begin{equation}\label{eq:dtb}
|\cD_t F^\epsilon(x)|\leq C_{\epsilon} \big(1+\nnrm{x}{D([0,t],\bR^d)}^q\big)
\end{equation}
for all $t\in[0,T)$ and $x\in D([0,t],\bR^d)$, where $C_{\epsilon},q\in(0,\infty)$ do not depend on $t$ or $x$.
\end{lemma}

\begin{proof}
Fix $t\in[0,T)$. In order to verify that $F^\epsilon$ is horizontally differentiable at $t$, we have to show that for every $x=x_t\in D([0,t],\bR^d)$ the right derivative
\begin{equation}\label{eq:proofExDtFeps1}
\cD_t F^\epsilon(x)=\frac{\dl^+}{\dl h}\bE f^\epsilon\big(X^{t+h,x_{t,h}}_T\big)\Big|_{h=0}
=\lim_{h\searrow0}\frac 1h\bE\big[f^\epsilon\big(X^{t+h,x_{t,h}}_T\big)-f^\epsilon\big(X^{t,x_{t}}_T\big)\big]
\end{equation}
exists. For $x\in D([0,t],\bR^d)$ and $y\in D([t,T],\bR^d)$, let $x\oplus y\in D([0,T],\bR^d)$ denote the càdlàg function defined by
\[x\oplus y\,(s):=
\begin{cases}
x(s),& s\in[0,t)\\
y(s),& s\in[t,T].
\end{cases}
\]
Moreover, for $h\in[0,T-t]$, let $T_h\colon D([t,T],\bR^d)\to D([t,T],\bR^d)$ be the translation operator defined by
\[(T_h y)(s):=
\begin{cases}
y(t),& s\in[t,t+h)\\
y(s-h),& s\in[t+h,T].
\end{cases}
\]
Note that, due to the time-homogeneity of Eq.~\eqref{eq:SDE}, the $D([0,T],\bR^d)$-valued random variables $X^{t+h,x_{t,h}}_T$ and $x_t\oplus T_h X^{t,x(t)}_T$ have the same distribution. As a consequence, we can rewrite \eqref{eq:proofExDtFeps1} as
\begin{equation}\label{eq:proofExDtFeps2}
\begin{aligned}
\cD_t F^\epsilon(x)=\frac{\dl^+}{\dl h}\bE f^\epsilon\big(x_t\oplus T_h X^{t,x(t)}_T\big)\Big|_{h=0}
&= \frac{\dl^+}{\dl h}\bE f\big(M^\epsilon\big[x_t\oplus T_h X^{t,x(t)}_T\big]\big)\Big|_{h=0}
\end{aligned}
\end{equation}
Now, for $y\in D([t,T],\bR^d)$ and $s\in[0,T]$,
\begin{equation}
\begin{aligned}
&M^\epsilon\big[x_t\oplus T_h y\big](s)\\
&=\int_{-\epsilon}^t\eta_\epsilon(s-r)\overline{x_t}(r)\,\dl r+\int_{t}^{t+h}\eta_\epsilon(s-r)y(t)\,\dl r+\int_{t+h}^{T}\eta_\epsilon(s-r)y(r-h)\,\dl r\\
&=\int_{-\epsilon}^t\eta_\epsilon(s-r)\overline{x_t}(r)\,\dl r+y(t)\int_{t}^{t+h}\eta_\epsilon(s-r)\,\dl r+\int_{t}^{T-h}\eta_\epsilon(s-r-h)y(r)\,\dl r
\end{aligned}
\end{equation}
and therefore, as $\supp\eta_\epsilon\subset [0,\epsilon]$ and $s\in [0,T]$ (and hence the boundary term vanishes when differentiating the third integral above),
\begin{equation*}
\begin{aligned}
\frac{\dl^+}{\dl h}M^\epsilon\big[x_t\oplus T_h y\big](s)&=\eta_{\epsilon}(s-t-h)y(t)-\int_{t}^{T-h}\eta_{\epsilon}'(s-r-h)y(r)\,\dl r\\
&=\eta_{\epsilon}(s-t-h)y(t)-\int_{t+h}^{T}\eta_{\epsilon}'(s-r)y(r-h)\,\dl r,~s\in [0,T].
\end{aligned}
\end{equation*}
The above calculation is also valid uniformly with respect to $s\in [0,T]$; that is, in $C([0,T],\bR^d)$, as $\eta$ is $C^{\infty}$ with compact support.
In order to differentiate under the expectation sign in \eqref{eq:proofExDtFeps2}, for $h\in[0,T-t]$, we have the bound
\begin{equation}\label{eq:dti}
\begin{aligned}
&\left|\frac{\dl^+}{\dl h} f\big(M^\epsilon\big[x_t\oplus T_h X^{t,x(t)}_T\big]\big)\right|\\
&=\left|Df\big(M^\epsilon\big[x_t\oplus T_h X^{t,x(t)}_T\big]\big)\,\big(x(t)\eta_\epsilon(\cdot -t-h)-\int_{t+h}^{T}\eta_{\epsilon}'(\cdot-r)X^{t,x(t)}(r-h)\,\dl r\big)\right|\\
&\le C\left(1+\left\|x_t\oplus T_h X^{t,x(t)}_T\right\|_{D([0,T],\bR^d)}\right)^{q'}\times\\
&\quad\quad\left\| x(t)\eta_\epsilon(\cdot -t-h)-\int_{t+h}^{T}\eta_{\epsilon}'(\cdot-r)X^{t,x(t)}(r-h)\,\dl r\right\|_{C([0,T],\bR^d)}\\
&\le C_{\epsilon}\left(1+\|x_t\|_{D([0,t],\bR^d)}+\left\|X^{t,x(t)}_{T}\right\|_{C([t,T],\bR^d)}\right)^{q'}\left\|X^{t,x(t)}_{T}\right\|_{C([t,T],\bR^d)}
\end{aligned}
\end{equation}
where the last upper bounds belongs to $L^p(\Omega)$ for every $p\in [1,\infty)$ due to \eqref{eq:standardEstInitialCond1}. Therefore, by \eqref{eq:proofExDtFeps2}, it follows that
\begin{equation}
\begin{aligned}
\cD_t F^\epsilon(x)&=
\frac{\dl^+}{\dl h}\bE f\big(M^\epsilon\big[x_t\oplus T_h X^{t,x(t)}_T\big]\big)\Big|_{h=0}=\bE\frac{\dl^+}{\dl h} f\big(M^\epsilon\big[x_t\oplus T_h X^{t,x(t)}_T\big]\big)\Big|_{h=0}\\
&=\bE \big[ Df(M^\epsilon X^{t,x}_{T})\,\big(x(t)\eta_\epsilon(\cdot -t)-\int_{t}^{T}\eta_{\epsilon}'(\cdot-r)X^{t,x}(r)\,\dl r\big],
\end{aligned}
\end{equation}
for $t\in [0,T)$ and $x\in D([0,t],\bR^d)$. The continuity of $\cD F^\epsilon$ at fixed times now follow from the formula \eqref{eq:dtf} and the continuity of $\xi\mapsto X^{t,\xi}_T$ asserted by Theorem~\ref{thm:regularityInitialCond}(i).
Finally, \eqref{eq:dtb} follows from \eqref{eq:dti} and \eqref{eq:standardEstInitialCond1} and therefore the extension $(\cD_t F^\epsilon)_{t\in[0,T]}$ of $(\cD_t F^\epsilon)_{t\in[0,T)}$ by zero belongs to the class $\bB([0,T])$.
\end{proof}

\begin{remark}\label{rem:dtnab}
Using the formulae for $\nabla_xF^{\epsilon}$ and $\nabla_x^2F^{\epsilon}$ from Lemmata \ref{lem:regularityFepsilon2} and \ref{lem:regularityFepsilon3}, respectively, and arguments completely analogous to the ones in Lemma \ref{lem:regularityFepsilon4} one also has that $\nabla_xF^{\epsilon}$ and $\nabla_x^2F^{\epsilon}$ are horizontally differentiable, if
$f\in C^2_p(C([0,T],\bR^d),\bR)$ and $f\in C^3_p(C([0,T],\bR^d),\bR)$, respectively (in fact, $\nabla_x^nF^{\epsilon}$ is horizontally differentiable, if $f\in C^{n+1}_p(C([0,T],\bR^d),\bR)$ for all $n\in \bN$). For $n=1,2$ the horizontal derivative $\cD \nabla_x^nF^\epsilon=(\cD_t \nabla_x^n F^\epsilon)_{t\in[0,T)}$ is continuous at fixed times, and the extension $(\cD_t \nabla_x^nF^\epsilon)_{t\in[0,T]}$ of $(\cD_t \nabla_x^n F^\epsilon)_{t\in[0,T)}$ by zero belongs to the class $\bB([0,T])$. Moreover,
\begin{equation*}
|\cD_t \nabla_x^n F^\epsilon(x)|\leq C_{\epsilon} \big(1+\nnrm{x}{D([0,t],\bR^d)}^q\big)
\end{equation*}
for all $t\in[0,T)$ and $x\in D([0,t],\bR^d)$, where $C_{\epsilon},q\in(0,\infty)$ do not depend on $t$ or $x$. For example, using that the $D([0,T],\bR^d)\times D([0,T],\bR^d)$-valued random variables $$\left(X^{t+h,x_{t,h}}_T, \one_{[t+h,T]}D^{e_i}X^{t+h,x_{t,h}}\right)\text{ and }\left(x_t\oplus T_hX_T^{t,x(t)},T_h(\one_{[t,T]}D^{e_i}\!X^{t,x(t)}_T)\right)$$ have the same distribution one can calculate, as in Lemma \ref{lem:regularityFepsilon4},
\begin{equation*}
\begin{aligned}
&(\mathcal{D}_t\nabla_xF^{\epsilon})_i(x)=\\
&\bE\big[D^2f(M^{\epsilon}X^{t,x}_T)\,\big(x(t)\eta_\epsilon(\cdot -t)-\int_{t}^{T}\eta_{\epsilon}'(\cdot-r)X^{t,x(t)}_{T}(r)\,\dl r,M^{\epsilon}[\one_{[t,T]}D^{e_i}\!X^{t,x(t)}_T]\big)\big]\\
&+\bE \big[ Df(M^\epsilon X^{t,x}_{T})\,\big(e_i\eta_\epsilon(\cdot -t)-\int_{t}^{T}\eta_{\epsilon}'(\cdot-r)D^{e_i}X^{t,x(t)}_{T}(r)\,\dl r\big)\big].
\end{aligned}
\end{equation*}
Furthermore, using the formula for $\cD F$ from Lemma \ref{lem:regularityFepsilon4} and arguments analogous to those in the proof of Lemmata \ref{lem:regularityFepsilon2} and \ref{lem:regularityFepsilon3} one can explicitly check, for $n=1,2$, that $\cD F$ is $n$-times vertically differentiable if $f\in C^{n+1}_p(C([0,T],\bR^d),\bR)$ and $\nabla_x^n\cD F=\cD \nabla^n_x F$ (in fact this holds for general $n\in \bN$).
\end{remark}

In summary, the combination of Lemmata \ref{lem:regularityFepsilon1}, \ref{lem:regularityFepsilon2}, \ref{lem:regularityFepsilon3} and \ref{lem:regularityFepsilon4} implies the desired regularity of $F^\epsilon$.

\begin{theorem}\label{thm:regularityFepsilon}
If $f\in C^2_p(C([0,T],\bR^d),\bR)$ and $\epsilon>0$, the non-anticipative functional $F^\epsilon=(F^\epsilon_t)_{t\in[0,T]}$ defined by \eqref{eq:defFteps} belongs to the class $\bC^{1,2}_b([0,T])$. The vertical and horizontal derivatives are given by \eqref{eq:nablaF}, \eqref{eq:nanab} and \eqref{eq:dtf}.
\end{theorem}

\section{Functional Kolmogorov equation}\label{sec:fk}
In this section we show that $F^{\epsilon}$ satisfies a backward functional Kolmogorov equation. We have already seen in the previous section that $F^{\epsilon}$ is regular enough when $f$ is. Therefore, in order to apply Theorem \ref{thm:functionalKolmogorov} one needs to check whether $(F_t^{\epsilon}(X_t))_{t\in [0,T]}$ is a martingale w.r.t.\ $(\cF_t)_{t\in[0,T]}$. This is easily done using the following result.
\begin{proposition}\label{prop:martingaleProperty}
Let $\varphi\colon D([0,T],\bR^d)\to\bR$ be a measurable mapping with polynomial growth and $\Phi=(\Phi_t)_{t\in[0,T]}$ be the non-anticipative functional defined by
\begin{equation}\label{eq:defPhit}
\Phi_t(x):=\bE\,\varphi(X^{t,x}_T),\qquad x\in D([0,t],\bR^d).
\end{equation}
 Then $(\Phi_t(X_t))_{t\in[0,T]}$ is a martingale w.r.t.\ $(\cF_t)_{t\in[0,T]}$.
\end{proposition}

\begin{proof}
The solution $X$ to Eq.~\eqref{eq:SDE} is a Markov process w.r.t.\ the filtration $(\cF_t)_{t\in[0,T]}$, see, e.g., \cite[Section 19.7]{SchPar14}.
For $0\leq s\leq t\leq T$, $x\in\bR$ and $\psi:D([t,T],\bR^d)\to\bR$ bounded and measurable we have
\begin{equation}\label{eq:MP}
\bE\big(\psi\big(X^{s,x}|_{[t,T]}\big)\big|\cF_t\big)=\bE\big(\psi\big(X^{t,y}|_{[t,T]}\big)\big)\big|_{y=X^{s,x}(t)},
\end{equation}
compare~\cite[Proposition 5.15]{KarShr98}.

Fix $0\leq s\leq t\leq T$ and assume for a moment that $\varphi$ is of the form
\[\varphi(x)=\varphi_1(x|_{[0,s]})\,\varphi_2(x|_{[s,t]})\,\varphi_3(x|_{[t,T]}),\quad x\in D([0,T],\bR^d),\]
with $\varphi_1:D([0,s],\bR^d)\to \bR$, $\varphi_2:D([s,t],\bR^d)\to\bR$ and $\varphi_3:D([t,T],\bR^d)\to\bR$ measurable and bounded. In this case,
\begin{equation}\label{eq:MG1}
\begin{aligned}
\bE(\Phi_t(X_t)|\cF_s)
&= \bE\big(\bE(\varphi(X^{t,y}))|_{y=X_t}\big|\cF_s\big)\\
&= \bE\Big[\varphi_1\big(X|_{[0,s]}\big)\,\varphi_2\big(X|_{[s,t]}\big)\,\bE\big(\varphi_3\big(X^{t,y}|_{[t,T]}\big)\big)\big|_{y=X(t)}\,\Big|\,\cF_s\Big]\\
&= \varphi_1\big(X|_{[0,s]}\big)\,\bE\Big[\varphi_2\big(X^{s,x}|_{[s,t]}\big)\,\bE\big(\varphi_3\big(X^{t,y}|_{[t,T]}\big)\big)\big|_{y=X^{s,x}(t)}\Big]\Big|_{x=X(s)}\\
&= \varphi_1\big(X|_{[0,s]}\big)\,\bE\Big[\varphi_2\big(X^{s,x}|_{[s,t]}\big)\,\bE\big(\varphi_3\big(X^{s,x}|_{[t,T]}\big)\big|\cF_t\big)\Big]\Big|_{x=X(s)}\\
&= \varphi_1\big(X|_{[0,s]}\big)\,\bE\Big[\varphi_2\big(X^{s,x}|_{[s,t]}\big)\,\varphi_3\big(X^{s,x}|_{[t,T]}\big)\Big]\Big|_{x=X(s)}\\
&= \bE\big(\varphi(X^{s,x})\big)\big|_{x=X_s}\\
&= \Phi_s(X_s),
\end{aligned}
\end{equation}
where we have used the Markov property \eqref{eq:MP} in the third and the fourth step.

Let $\mathcal C$ denote the collection of all cylinder sets $A\in\cB_T$ of the form
\begin{equation*}
A=\{x\in D([0,T],\bR^d): x(t_1)\in B_1,\ldots, x(t_n)\in B_n\}
\end{equation*}
where $0\leq t_1\leq \ldots\leq t_n\leq T$, $B_i\in\cB(\bR^d)$, $i=1,\ldots,n$, and $n\in\bN$.
Then $\cC$ is closed under finite intersections, $\sigma(\cC)=\cB_T$, and all $A\in\cC$ satisfy
\begin{equation}\label{eq:MG2}
\bE\big(\bE(\one_A(X^{t,y}))|_{y=X_t}\big|\cF_s\big)=\bE\big(\one_A(X^{s,x})\big)\big|_{x=X_s}
\end{equation}
according to \eqref{eq:MG1} with $\varphi=\one_A$. Since the class of all $A\in\cB_T$ satisfying \eqref{eq:MG2} is a Dynkin system, we obtain that \eqref{eq:MG2} is fulfilled for all sets $A\in\cB_T$. By approximation, the indicator function $\one_A$ in \eqref{eq:MG2} can be replaced by every measurable $\varphi:D([0,T],\bR^d)\to\bR$ with polynomial growth.
\end{proof}
Now the backward functional Kolmogorov equation for $F^\epsilon$ follows almost immediately.
\begin{corollary}\label{cor:kol}
If $f\in C^2_p(C([0,T],\bR^d),\bR)$ and $\epsilon>0$, then the non-anticipative functional $F^\epsilon=(F^\epsilon_t)_{t\in[0,T]}$ defined by \eqref{eq:defFteps} satisfies the functional partial differential equation
\begin{equation}\label{eq:backwardKolm}
\left.
\begin{aligned}
\cD_t F^\epsilon(x_t)&=-b(x(t))\nabla_x F^\epsilon_t(x_t)-\frac12\tr\big(\nabla_x^2 F^\epsilon_t(x_t)\,\sigma(x(t))\,\sigma^\top(x(t))\big)\\
F^\epsilon_T(x) &= f(x)
\end{aligned}
\quad\right\}
\end{equation}
for all $t\in(0,T)$ and all $x\in C([0,T],\bR^d)$ with $x(0)=\xi_0$.
\end{corollary}
\begin{proof}
It follows from Theorem \ref{thm:regularityFepsilon} that $F^\epsilon\in \bC_b^{1,2}([0,T])$ and Proposition \ref{prop:martingaleProperty} shows that $(F_t^{\epsilon}(X_t))_{t\in [0,T]}$ is a martingale w.r.t.\ $(\cF_t)_{t\in[0,T]}$. As shown in Lemma \ref{lem:topsup}, the topological support of $X$ in $C([0,T],\bR^d)$ is the set $\{x\in C([0,T],\bR^d):~x(0)=\xi_0\}$ and hence the result follows from Theorem \ref{thm:functionalKolmogorov}.
\end{proof}

\section{Error representation}\label{sec:errep}
Here we give an explicit formula for the weak error $\bE (f^\epsilon(\tilde X_T)-f^\epsilon(X_T))$, where $X$ and $\tilde X$ are the solutions to \eqref{eq:SDE} and \eqref{eq:SDEtilde}, respectively, and $f^\varepsilon$ is the regularized version of a given path-dependent functional $f$ as defined in \eqref{eq:deffepsilon}--\eqref{eq:defMepsilon2}. 
As the following remark shows, we implicitly also obtain a representation of the weak error $\bE (f(\tilde X_T)-f(X_T))$ for the `original' functional $f$.

\begin{remark}\label{rem:fepf}
Under Assumptions~\ref{ass:bsigma} and \ref{ass:bsigmatilde} and for $f\in C^1_p(C([0,T],\bR^d),\bR)$, we have 
\begin{equation}\label{eq:Euler_8}
\bE\big(f(\tilde X_T)-f(X_T)\big)=\lim_{\epsilon\to0}\bE\big(f^\epsilon(\tilde X_T)-f^\epsilon(X_T)\big).
\end{equation}
This follows from applying a first order Taylor expansion to $f$ around $X_T$,
the dominated convergence theorem, using that $f^\epsilon(x)\xrightarrow{\epsilon\to0}f(x)$ for all $x\in C([0,T],\bR)$ and the finiteness of $\bE(\|\tilde X_T\|^p_{C([0,T],\bR)}+\| X_T\|^p_{C([0,T],\bR)})$ for $p\ge 1$.
\end{remark}

The proof of our error representation formula is based on the functional Itô formula from Theorem~\ref{thm:functionalIto}, the 
regularity properies of the non-anticipative functional $F^\varepsilon$ and the explicit representation of its derivatives from Theorem~\ref{thm:regularityFepsilon}, and the backward functional Kolmogorov equation from Corollary~\ref{cor:kol}. 
Recall that  we assume $X(0)=\tilde X(0)=\xi_0\in\bR^d$ and hence
by the definition \eqref{eq:defFteps} of
$F^\varepsilon$, we have
\begin{equation*}
\bE \big(f^\epsilon(\tilde X_T)-f^\varepsilon(X_T)\big)=\bE \big(F^\epsilon_T(\tilde X_T)-F^\epsilon_0(\tilde X_0)\big).
\end{equation*}

\begin{theorem}\label{thm:errorRepresentation}
Let Assumptions~\ref{ass:bsigma} and \ref{ass:bsigmatilde} hold, and let $X=(X(t))_{t\geq0}$ and $\tilde X=(\tilde X(t))_{t\in[0,T]}$ be the strong solutions to Equations~\eqref{eq:SDE} and \eqref{eq:SDEtilde}, respectively, both starting from $\xi_0\in\bR^d$. Let $f\in C^2_p(C([0,T],\bR^d),\bR)$ and, for $\epsilon>0$, let $f^\epsilon$ and  $F^\epsilon=(F^{\epsilon}_t)_{t\in[0,T]}$ be given by \eqref{eq:deffepsilon}--\eqref{eq:defMepsilon2} and \eqref{eq:defFteps}, respectively.
Then, the following weak error formula holds:
\begin{equation}\label{eq:errorRepresentationEpsilon1}
\begin{aligned}
&\bE\big(f^\epsilon(\tilde X_T)-f^\epsilon(X_T)\big)\\
&= \bE\Bigg(\int_0^T\nabla_xF^\epsilon_t(\tilde X_t)\; \big(\tilde b(t,\tilde X_t)-b(\tilde X(t))\big)\,\dl t \\
&\quad+\frac12\int_0^T\tr\Big\{\nabla_x^2 F^\epsilon_t(\tilde X_t)\big(\tilde \sigma(t,\tilde X_t)\,\tilde \sigma^\top(t,\tilde X_t)-\sigma(\tilde X(t))\,\sigma^\top(\tilde X(t))\big)\Big\}\,\dl t\Bigg).
\end{aligned}
\end{equation}
Writing the vertical derivatives of $F^\epsilon$ explicitly, this reads
\begin{equation}\label{eq:errorRepresentationEpsilon2}
\begin{aligned}
&\bE\big(f^\epsilon(\tilde X_T)-f^\epsilon(X_T)\big)\\
&= \bE\Bigg(\int_0^T\sum_{j=1}^d\big(\bE\big[Df^\epsilon(X^{t,x}_T)\,(\one_{[t,T]}D^{e_j}\!X^{t,x(t)}_T)\big]\big)\big|_{x=\tilde X_t}\, \big(\tilde b_j(t,\tilde X_t)-b_j(\tilde X(t))\big)\,\dl t \\
&\quad+\frac12\int_0^T\sum_{i,j,k=1}^d\Big\{\Big(\bE\Big[D^2f^\epsilon(X^{t,x}_T)\,\big(\one_{[t,T]}D^{e_i}X^{t,x(t)}_T,\,\one_{[t,T]}D^{e_j}X^{t,x(t)}_T\big)\\
&\quad+Df^\epsilon(X^{t,x})\,\big(\one_{[t,T]}D^{e_i+e_j}X^{t,x(t)}_T\big)\Big]\Big)\Big|_{x=\tilde X_t}\big(\tilde \sigma_{ik}\,\tilde \sigma_{jk}(t,\tilde X_t)-\sigma_{ik}\,\sigma_{jk}(\tilde X(t))\big)\Big\}\,\dl t\Bigg).
\end{aligned}
\end{equation}
\end{theorem}

\begin{proof}
By Theorem~\ref{thm:regularityFepsilon} we can apply the functional Itô formula (Theorem~\ref{thm:functionalIto}) to the non-anticipative functional $F^\epsilon=(F^\epsilon_t)_{t\in[0,T]}$ and the continuous semi-martingale $\tilde X=(\tilde X(t))_{t\in[0,T]}$. Therefore,
\begin{equation*}\label{eq:proofErrorRep1}
\begin{aligned}
F^\epsilon_T&(\tilde X_T)- F^\epsilon_0(\tilde X_0)\\
&= \int_0^T\cD_tF^\epsilon(\tilde X_t)\,\dl t
+\int_0^T\nabla_xF^\epsilon_t(\tilde X_t)\; \dl \tilde X(t)
+\frac12\int_0^T\tr\big(\nabla_x^2 F^\epsilon_t(\tilde X_t)\;\dl [\tilde X](t)\big)\\
&= \int_0^T\cD_tF^\epsilon(\tilde X_t)\,\dl t
+\int_0^T\nabla_xF^\epsilon_t(\tilde X_t)\; \big(\tilde b(t,\tilde X_t)\,\dl t +\tilde \sigma(t,\tilde X_t)\,\dl W(t)\big) \\
&\quad+\frac12\int_0^T\tr\big(\nabla_x^2 F^\epsilon_t(\tilde X_t)\,\tilde \sigma(t,\tilde X_t)\,\tilde \sigma^\top(t,\tilde X_t)\big)\,\dl t.
\end{aligned}
\end{equation*}
Using the functional backward Kolmogorov equation 
from Corollary~\ref{cor:kol} and taking expectations,
we obtain \eqref{eq:errorRepresentationEpsilon1}. The explicit formulas for the vertical derivatives of $F^\epsilon$ in Lemmata~\ref{lem:regularityFepsilon2} and \ref{lem:regularityFepsilon3} yield \eqref{eq:errorRepresentationEpsilon2}.
\end{proof}

\section{Application to the Euler scheme}\label{sec:euler}
In this section we consider the one-dimensional case $d=m=1$ and the explicit Euler discretization of \eqref{eq:SDE}. Let
$0=\tau_0<\tau_1<\ldots<\tau_N=T$ be discretization times with maximal step size \[\delta:=\max\{|\tau_{n+1}-\tau_n|:n=1,\ldots,N\},\] and let
$(Y(\tau_n))_{n\in\{0,\ldots,N\}}$ be given by $Y(0)=\xi_0$ and
\begin{align*}
Y(\tau_{n+1})&=Y(\tau_n)+b(Y(\tau_n))(\tau_{n+1}-\tau_n)+\sigma(Y(t_n))(W(\tau_{n+1})-W(\tau_n)).
\end{align*}
Let $(Y(t))_{t\in[0,T]}$ be the continuous-time process obtained by piecewise linear interpolation of $(Y(\tau_n))_{n\in\{0,\ldots,N\}}$; i.e., for $n\in\{0,\ldots,N-1\}$ and $t\in[\tau_n,\tau_{n+1}]$, we define
\begin{equation}\label{eq:pwlinearInterpol}
\begin{aligned}
Y(t)
&= Y(\tau_n) + \frac{t-\tau_n}{\tau_{n+1}-\tau_n}(Y(\tau_{n+1})-Y(\tau_n))\\
&= Y(\tau_n) + \int_{\tau_n}^tb(Y(\tau_n))\,\dl s + \int_{\tau_n}^t\sigma(Y(\tau_n))(W(\tau_{n+1})-W(\tau_n))\,\dl s.
\end{aligned}
\end{equation}
Our main result of this section is as follows. It is a direct consequence of Proposition \ref{prop:yt} and Proposition \ref{prop:Euler_error_2}, both of which are proved subsequently, and the triangle inequality.
\begin{theorem}\label{thm:euler}
Let Assumption~\ref{ass:bsigma} hold with $d=m=1$. Let $(X(t))_{t\geq0}$ be the strong solution to \eqref{eq:SDE} and $(Y(t))_{t\in[0,T]}$, given by \eqref{eq:pwlinearInterpol}, be the piecewise linear interpolation of the solution to the explicit Euler scheme applied to \eqref{eq:SDE}. If $f\in C^4_p(C([0,T],\bR),\bR)$, 
then there exists a constant $C\in(0,\infty)$ which does not depend on the maximal step size $\delta$ such that, for all $\delta\in (0,1]$,
\begin{align*}
\big|\bE\big(f(Y_T)-f(X_T)\big)\big|\leq C \delta.
\end{align*}
\end{theorem}
Note that while $Y$ is numerically computable it does not satisfy an equation like \eqref{eq:SDEtilde} and hence the weak error representation from Theorem \ref{thm:errorRepresentation} is not directly applicable. Therefore, we will first define a stochastic interpolation $(\tilde X(t))_{t\in[0,T]}$ of $(Y(\tau_n))_{n\in\{0,\ldots,N\}}$, given below by \eqref{eq:stochasticInterpol}, which is not feasible for numerical computations but satisfies an SDE of the type \eqref{eq:SDEtilde}. Then we have
\begin{equation}\label{eq:errordecomb}
\bE\big(f(Y_T)-f(X_T)\big)=\bE\big(f(Y_T)-f(\tilde X_T)\big)+\bE\big(f(\tilde X_T)-f(X_T)\big).
\end{equation}
The two terms on the right-hand side will be analysed in the following two subsections. The first term is easier to handle and will be treated by means of a second order Taylor expansion of $f$ around $Y_T$ and a Lévy-Ciesielsky-type expansion of Brownian motion (no functional Itô calculus arguments are used here). The more difficult estimation of the second term on the right hand side of \eqref{eq:errordecomb} is based on our general error expansion result in Theorem~\ref{thm:errorRepresentation}.

As an application of Theorem~\ref{thm:euler} we consider the approximation of covariances $\operatorname{Cov}(X(t_1),X(t_2))$ of the solution process.
\begin{example}\label{ex:covariance}
Let $t_1,t_2\in[0,T]$. In the situation of Theorem~\ref{thm:euler} we have that
\begin{equation}\label{eq:estCov}
|\operatorname{Cov}(Y(t_1),Y(t_2))-\operatorname{Cov}(X(t_1),X(t_2))|\leq C\delta
\end{equation}
for all $\delta\in(0,1]$, with a constant $C\in(0,\infty)$ independent of $\delta$. Indeed, note that
\begin{align*}
|\operatorname{Cov}(Y(t_1),Y(t_2))-\operatorname{Cov}(X(t_1),X(t_2))|
&\leq |\bE(Y(t_1)Y(t_2))-\bE(X(t_1)X(t_2))|\\
&\quad+|\bE (Y(t_1))-\bE(X(t_1))|\cdot|\bE(Y(t_2))|\\
&\quad+|\bE(X(t_1))|\cdot|\bE(Y(t_2))-\bE(X(t_2))|.
\end{align*}
Since $\bE(Y(t_1))$ is bounded independently of $\delta$, the estimate \eqref{eq:estCov} follows from three applications of Theorem~\ref{thm:euler} to the functionals $f_0,f_1,f_2\colon C([0,T],\bR)\to\bR$ given by
\[f_0(x)=x(t_1)x(t_2),\quad f_1(x)=x(t_1),\quad f_2(x)=x(t_2).\]
\end{example}

\subsection{From piecewise linear to stochastic interpolation}

For $t\in[0,T]$ we use the notation
\begin{align*}
\tau_n(t):=\max\{\tau_m:m\in\{0,\ldots,N\},\;\tau_m\leq t\},
\end{align*}
and for $x_t\in D([0,t];\bR)$ we set
\begin{align}\label{eq:tildebtildesigma}
\tilde b(t,x_t):=b(x(\tau_n(t))),\qquad \tilde\sigma(t,x_t):=\sigma(x(\tau_n(t))).
\end{align}
Let $(\tilde X(t))_{t\in[0,T]}$ be the stochastic interpolation of $(Y(\tau_n))_{n\in\{0,\ldots,N\}}$ given by \eqref{eq:SDEtilde} with $\tilde b$ and $\tilde\sigma$ defined by \eqref{eq:tildebtildesigma}. That is, for $n\in\{0,\ldots,N-1\}$ and $t\in[\tau_n,\tau_{n+1}]$,
\begin{align}\label{eq:stochasticInterpol}
\tilde X(t)
= Y(\tau_n) + \int_{\tau_n}^tb(Y(\tau_n))\,\dl s + \int_{\tau_n}^t\sigma(Y(\tau_n))\,\dl W(s).
\end{align}

\begin{proposition}\label{prop:yt}
Let Assumption~\ref{ass:bsigma} hold with $d=m=1$. Let $(Y(t))_{t\in[0,T]}$ be the piecewise linear interpolation of the solution to the explicit Euler scheme given by \eqref{eq:pwlinearInterpol} and $(\tilde X(t))_{t\in[0,T]}$ be the corresponding stochastic interpolation given by \eqref{eq:stochasticInterpol}.
If $f\in C_p^2(C([0,T],\bR),\bR)$,
then there exists a constant $C\in(0,\infty)$ not depending on $\delta$ such that, for all $\delta\in (0,1]$,
\begin{align*}
\big|\bE\big(f(Y_T)-f(\tilde X_T)\big)\big|\leq C \delta.
\end{align*}
\end{proposition}

\begin{proof}
A second order Taylor expansion of $f$ around $Y_T$ yields
\begin{equation}\label{eq:pwlinear1}
\begin{aligned}
\bE\big(f(\tilde X_T)-f(Y_T)\big)&=\bE\big( Df(Y_T)(\tilde X_T-Y_T)\big)\\
&\quad +\bE\Big((1-\theta)\int_0^1 D^2f\big(Y_T+\theta(\tilde X_T-Y_T)\big)\big(\tilde X_T-Y_T,\tilde X_T-Y_T\big)\,\dl\theta\Big)\\
&=:e_1+e_2.
\end{aligned}
\end{equation}
We show that the first term $e_1$ on the right hand side of \eqref{eq:pwlinear1} equals zero. This follows from the fact that the $C([0,T],\bR)$-valued random variables $\tilde X_T-Y_T$ and $Y_T$ are independent, that $\|Y_T\|_{C([0,T],\bR)}$ has finite moments of all orders uniformly in $\delta$ (this can be easily seen from \eqref{eq:pwlinearInterpol}), and that the $C([0,T],\bR)$-valued random variable $\tilde X_T-Y_T$ is integrable and has mean zero. To see the latter, observe that
in view of \eqref{eq:pwlinearInterpol} and \eqref{eq:stochasticInterpol} we have
\begin{equation}\label{eq:tildeX-Y}
\tilde X(t)-Y(t)=\sum_{n=0}^{N-1}\one_{(\tau_n,\tau_{n+1}]}(t)\Big(\big(W(t)-W(\tau_n)\big)-\frac{t-\tau_n}{\tau_{n+1}-\tau_n}\big(W(\tau_{n+1})-W(\tau_n)\big)\Big).
\end{equation}
In order to verify the independence of $\tilde X_T-Y_T$ and $Y_T$, we use a suitable modification of the Lévy-Ciesielski construction of Brownian motion.
Let $(H_k)_{k\in\bN_0}$ be the Haar orthonormal basis of $L^2([0,1];\bR)$, i.e., $H_0(t)=1$ and for $j\in\bN$ and $\ell\in\{0,\ldots,2^j-1\}$
\begin{align*}
H_{2^j+\ell}(t)=
\begin{cases}
2^{j/2},&\text{ on }\big[\frac{\ell}{2^j},\frac{2\ell+1}{2^{j+1}}\big)\\
-2^{j/2},&\text{ on } \big[\frac{2\ell+1}{2^{j+1}},\frac{\ell+1}{2^j}\big)\\
0,&\text{ otherwise.}
\end{cases}
\end{align*}
For every $n\in\{0,\ldots,N-1\}$ we define a corresponding orthonormal basis $(H^n_k)_{k\in\bN_0}$ of $L^2([\tau_n,\tau_{n+1}];\bR)$ by setting
\begin{align*}
H^n_k(x):=(\tau_{n+1}-\tau_n)^{-1/2}H_k\Big(\frac{t-\tau_n}{\tau_{n+1}-\tau_n}\Big),\quad t\in[\tau_n,\tau_{n+1}].
\end{align*}
The Schauder functions corresponding to the $H^n_k$ are denoted by $S^n_k$, i.e.,
$S^n_k(t):=\int_{\tau_n}^t H^n_k(s)\,\dl s$, $t\in[\tau_n,\tau_{n+1}]$. In the sequel, we identify the Haar and  Schauder functions $H^n_k$ and $S^n_k$ with their extensions by zero to $[0,T]$.
Arguing as in the proof of the Lévy-Ciesielski construction of Brownian motion (see, e.g., \cite{SchPar14}) we have
\begin{align*}
W|_{[\tau_n,\tau_{n+1}]}=\sum_{k=0}^\infty \big(\int_{\tau_n}^{\tau_{n+1}} H^n_k(s)\,\dl W(s)\big)\,S^n_k
\end{align*}
as an identity in the space $L^2(\Omega;C([\tau_n,\tau_{n+1}];\bR))$, where the infinite sum converges in $L^2(\Omega;C([\tau_n,\tau_{n+1}];\bR))$.
This yields the representation
\begin{align*}
W_T=\sum_{k=0}^\infty \Big\{\sum_{n=0}^{N-1}\big(\int_{\tau_n}^{\tau_{n+1}} H^n_k(s)\,\dl W(s)\big)\,S^n_k\,\one_{(\tau_n,\tau_{n+1}]}\Big\},
\end{align*}
holding as an identity in the space $L^2(\Omega;C([0,T];\bR))$.
Note that the random variables $\int_{\tau_n}^{\tau_{n+1}}H^n_k\,\dl W(s)$, $n\in\{0,\ldots,N-1\}$, $k\in\bN_0$ are independent and standard normally distributed.
By \eqref{eq:tildeX-Y} and the fact that each family $(S^n_k)_{k\in\bN_0}$ is a Schauder basis for $C([\tau_n,\tau_{n+1}],\bR)$, it is now obvious that
\begin{align*}
\tilde X_T-Y_T=\sum_{k=1}^\infty \Big\{\sum_{n=0}^{N-1}\big(\int_{\tau_n}^{\tau_{n+1}} H^n_k(s)\,\dl W(s)\big)\,S^n_k\,\one_{(\tau_n,\tau_{n+1}]}\Big\},
\end{align*}
where the infinite sum starts at $k=1$ instead of $k=0$.
Since $Y_T$ can be represented as a functional of the random variables $\int_{\tau_n}^{\tau_{n+1}}H^n_0\,\dl W(s)$, $n\in\{0,\ldots,N-1\}$, it follows that the $C([0,T],\bR)$-valued random variables $\tilde X_T-Y_T$ and $Y_T$ are independent.

It remains to estimate the absolute value of the second term on the right hand side of \eqref{eq:pwlinear1}. As the second derivative of $f$ has polynomial growth, we use H\"older's inequality to estimate
$$
|e_2|\le C\left(\bE\|\tilde{X}_T\|^{2p}_{C([0,T],\bR)}+\bE\|Y_T\|^{2p}_{C([0,T],\bR)}\right)^{\frac12}\left(\bE\big(\|\tilde X_T-Y_T\|^4_{C([0,T],\bR)}\big)\right)^{\frac12}
$$
Using Gronwall's lemma and the Burkolder inequality one can check that $\bE\|\tilde{X}_T\|^{2p}_{C([0,T],\bR)}$ and $\bE\|Y_T\|^{2p}_{C([0,T],\bR)}$ are bounded uniformly in $\delta$.
Finally, using \eqref{eq:tildeX-Y}, we have
\begin{align*}
&\bE\big(\|\tilde X_T-Y_T\|^4_{C([0,T],\bR)}\big)=\bE\Big(\sup_{t\in[0,T]}(\tilde X(t)-Y(t))^4\Big)\\
&\leq 8\,\bE\Big(\sup_{t\in[0,T]}\sum_{n=0}^{N-1}\one_{(\tau_n,\tau_{n+1}]}(t)(W(t)-W(\tau_n))^4\Big)+8\,\bE\Big(\sup_{n\in\{0,\ldots,N-1\}}(W(\tau_{n+1})-W(\tau_n))^4\Big)\\
&\leq 8\left(\frac43\right)^4\sup_{t\in[0,T]}\bE\Big(\sum_{n=0}^{N-1}\one_{(\tau_n,\tau_{n+1}]}(t)(W(t)-W(\tau_n))^4\Big)\\
&\qquad+8\left(\frac43\right)^4\sup_{n\in\{0,\ldots,N-1\}}\bE\big((W(\tau_{n+1})-W(\tau_n))^4\big)\\
&= 48\left(\frac43\right)^4\,\delta^2,
\end{align*}
where, in the penultimate step, we have used Doob's maximal inequality for submartingales.
\end{proof}

\subsection{Weak order for the stochastically interpolated Euler scheme}

Here we use our main result, Theorem~\ref{thm:errorRepresentation}, to estimate the second term on the right hand side of \eqref{eq:errordecomb}.

\begin{proposition}\label{prop:Euler_error_2}
Let Assumption~\ref{ass:bsigma} hold with $d=m=1$. Let $(X(t))_{t\geq0}$ be the strong solution to \eqref{eq:SDE} and $(\tilde X(t))_{t\in[0,T]}$ be the solution to the stochastically interpolated Euler scheme given by \eqref{eq:stochasticInterpol}.
If $f\in C^4_p(C([0,T],\bR),\bR)$, then there exists a constant $C\in(0,\infty)$ not depending on $\delta$ such that, for all $\delta\in(0,1]$,
\begin{align*}
\big|\bE\big(f(\tilde X_T)-f(X_T)\big)\big|\leq C \delta.
\end{align*}
\end{proposition}

We prepare the proof of Proposition~\ref{prop:Euler_error_2} by proving three Lemmata.
Note in particular that Lemma~\ref{lem:Euler2} states a functional backward Kolmogorov equation for the vertical derivatives of $F^\epsilon$. In the sequel, Assumption~\ref{ass:bsigma} is supposed to hold for $d=m=1$, and $f^\epsilon$ and  $F^\epsilon$ are given by \eqref{eq:deffepsilon}--\eqref{eq:defMepsilon2} and \eqref{eq:defFteps}, respectively. Moreover, we use the following notation, similar to the one used in the proof of Lemma~\ref{lem:regularityFepsilon4}: Given $0\leq\tau\leq t\leq T$ and $x\in D([0,\tau],\bR)$, $y\in D([\tau,t],\bR)$, we write $x\oplus y\in D([0,t],\bR)$ for the càdlàg function defined by
\[x\oplus y\,(s):=
\begin{cases}
x(s),&s\in[0,\tau)\\
y(s),&s\in[\tau,t].
\end{cases}
\]

\begin{lemma}\label{lem:Euler1}
Let $f\in C^3_p(C([0,T],\bR),\bR)$ and fix $\epsilon>0$, $n\in\{0,\ldots,N-1\}$ and $x_{\tau_n}\in C([0,\tau_n];\bR)$. Let $G=(G_t)_{t\in[\tau_n,\tau_{n+1}]}$ be the non-anticipative functional on $D([\tau_n,\tau_{n+1}],\bR)$ defined by
\begin{equation}\label{eq:Euler2}
\begin{aligned}
G_t(y_t):=\nabla_xF_t^\epsilon(x_{\tau_n}\oplus y_t)\big(b(y(\tau_n))-b(y(t))\big),\quad y_t\in D([\tau_n,t];\bR).
\end{aligned}
\end{equation}
Then $G$ belongs to the class $\bC^{1,2}_{b}([\tau_n,\tau_{n+1}])$, and for $t\in[\tau_n,\tau_{n+1}]$ and $y_t\in D([\tau_n,t],\bR)$ we have
\begin{align*}
\cD_tG(y_t)&=(\cD_t \nabla_x F^\epsilon)(x_{\tau_n}\oplus y_t)\big(b(y(\tau_n))-b(y(t))\big),\\
\nabla_xG_t(y_t)&=\nabla_x^2 F_t^\epsilon(x_{\tau_n}\oplus y_t)\big(b(y(\tau_n))-b(y(t))\big)+\nabla_x F_t^\epsilon(x_{\tau_n}\oplus y_t)\,b'(y(t)),\\
\nabla_x^2G_t(y_t)&=\nabla_x^3 F_t^\epsilon(x_{\tau_n}\oplus y_t)\big(b(y(\tau_n))-b(y(t))\big)
+2\nabla_x^2 F_t^\epsilon(x_{\tau_n}\oplus y_t)\,b'(y(t))\\
&\quad +\nabla_x F_t^\epsilon(x_{\tau_n}\oplus y_t)\,b''(y(t)).
\end{align*}
\end{lemma}
\begin{proof}
One easily checks that if  $H=(H_t)_{t\in[a,b]}$ and $K=(K_t)_{t\in [a,b]}$ are  non-anticipative functionals on $D([a,b],\bR)$, and both $H$ and $K$ are horizontally and vertically differentiable, then so is their product $HK=(H_tK_t)_{t\in[a,b]}$ and we have the product rules
$\mathcal{D}(HK)=H\mathcal{D}K+K\mathcal{D}H$ and $\nabla_x(HK)=H\nabla_x K+K\nabla_x H$. Therefore, since left-continuity implies continuity at fixed times, it follows that if $H,K\in \bC_b^{1,k}([a,b])$, then $HK\in \bC_b^{1,k}([a,b])$.  Define the functional $K=(K_t)_{t\in[\tau_n,\tau_{n+1}]}$ on  $D([\tau_n,\tau_{n+1}],\bR)$ by $K_t(y_t)=b(y(\tau_n))-b(y(t))$, $y_t\in D([\tau_n,t],\bR)$. It is immediate from the definitions that $\mathcal{D}K=0$ and that $\nabla^n_xK_t(y_t)=-b^{(n)}(y(t))$ and hence $K\in \bC_b^{1,k}([0,T])$. If one defines the functional $H=(H_t)_{t\in[\tau_n,\tau_{n+1}]}$ on  $D([\tau_n,\tau_{n+1}],\bR)$ by $H_t(y_t)=\nabla_xF_t^\epsilon(x_{\tau_n}\oplus y_t)$, $y_t\in D([\tau_n,t],\bR)$, then  $$\mathcal{D}_tH(y_t)=\mathcal{D}_t\nabla_xF_t^\epsilon(x_{\tau_n}\oplus y_t)\text{ and } \nabla^n_xH_t(y_t)=\nabla^{n+1}_xF_t^\epsilon(x_{\tau_n}\oplus y_t).$$ As $\nabla_xF^{\varepsilon}\in \bC_b^{1,2}([0,T])$ we have $H\in \bC_b^{1,2}([0,T])$ by Remarks \ref{rem:nabn} and \ref{rem:dtnab}, and the statement follows.
\end{proof}
A completely analogous argument gives the following result and therefore we omit the proof.
\begin{lemma}\label{lem:Euler1a}
Let $f\in C^4_p(C([0,T],\bR),\bR)$ and fix $\epsilon>0$, $n\in\{0,\ldots,N-1\}$ and $x_{\tau_n}\in C([0,\tau_n],\bR)$. Let $H=(H_t)_{t\in[\tau_n,\tau_{n+1}]}$ be the non-anticipative functional on $D([\tau_n,\tau_{n+1}],\bR)$ defined by
\begin{equation*}
\begin{aligned}
H_t(y_t):=\nabla^2_xF_t^\epsilon(x_{\tau_n}\oplus y_t)\big(\sigma^2(y(\tau_n))-\sigma^2(y(t))\big),\quad y_t\in D([\tau_n,t],\bR).
\end{aligned}
\end{equation*}
Then $H$ belongs to the class $\bC^{1,2}_{b}([\tau_n,\tau_{n+1}])$, and for $t\in[\tau_n,\tau_{n+1}]$ and $y_t\in D([\tau_n,t],\bR)$ we have
\begin{align*}
\cD_t H(y_t)&=(\cD_t \nabla^2_x F^\epsilon)(x_{\tau_n}\oplus y_t)\big(\sigma^2(y(\tau_n))-\sigma^2(y(t))\big),\\
\nabla_x H_t(y_t)&=\nabla_x^3 F_t^\epsilon(x_{\tau_n}\oplus y_t)\big(\sigma^2(y(\tau_n))-\sigma^2(y(t))\big)+2\nabla^2_x F_t^\epsilon(x_{\tau_n}\oplus y_t)\,(\sigma\sigma')(y(t)),\\
\nabla_x^2H_t(y_t)&=\nabla_x^4 F_t^\epsilon(x_{\tau_n}\oplus y_t)\big(\sigma^2(y(\tau_n))-\sigma^2(y(t))\big)
+4\nabla_x^3 F_t^\epsilon(x_{\tau_n}\oplus y_t)\,(\sigma\sigma')(y(t))\\
&\quad +2\nabla_x^2 F_t^\epsilon(x_{\tau_n}\oplus y_t)\,((\sigma')^2+\sigma\sigma'')(y(t)).
\end{align*}
\end{lemma}
\begin{lemma}\label{lem:Euler2}
Let $f\in C^{2+n}_p(C([0,T],\bR),\bR)$, $n=1,2$, and fix $\epsilon>0$, $n\in\{0,\ldots,N-1\}$ and $x_{\tau_n}\in C([0,\tau_n],\bR)$ with $x(0)=\xi_0$. For all $t\in(\tau_n,\tau_{n+1})$ and $y\in C([\tau_n,\tau_{n+1}],\bR)$ such that $y(\tau_n)=x(\tau_n)$  we have
\begin{align*}
\cD_t(\nabla_x^nF^\epsilon)(x_{\tau_n}\oplus y_t)=-\nabla_x^{n+1}F_t^\epsilon(x_{\tau_n}\oplus y_t)\,b(y(t))-\frac12\nabla_x^{n+2}F_t^\epsilon(x_{\tau_n}\oplus y_t)\,\sigma^2(y(t)).
\end{align*}
\end{lemma}
\begin{proof}
As discussed in Remark \ref{rem:dtnab} we have that $\cD\nabla^n_x F=\nabla_x^n\cD F$.  Hence, as $x_{\tau_n}\oplus y_t\in C([0,t],\bR)$ with $(x_{\tau_n}\oplus y_t)(0)=\xi_0$, the statement follows from Corollary \ref{cor:kol} by applying $\nabla_x$, respectively $\nabla_x^2$, to the functional Kolmogorov equation \eqref{eq:backwardKolm} and extending $x_{\tau_n}\oplus y$ 
continuously to $[0,T]$.
\end{proof}
We are now ready to verify the error estimate in Proposition~\ref{prop:Euler_error_2}.

\begin{proof}[Proof of Proposition~\ref{prop:Euler_error_2}]
Let $\epsilon>0$ be fixed. In view of Remark \ref{rem:fepf} it is enough to bound $\bE\big(f^\epsilon(\tilde X_T)-f^\epsilon(X_T)\big)$ independently of $\epsilon>0$. By Theorem~\ref{thm:errorRepresentation}, we have
\begin{equation}\label{eq:Euler_error_2}
\begin{aligned}
\bE\big(f^\epsilon(\tilde X_T)-f^\epsilon(X_T)\big)
&= \bE\int_0^T\nabla_xF^\epsilon_t(\tilde X_t)\; \big(\tilde b(t,\tilde X_t)-b(\tilde X(t))\big)\,\dl t\\
&\quad +\frac12\bE\int_0^T\nabla_x^2 F^\epsilon_t(\tilde X_t)\big(\tilde \sigma^2(t,\tilde X_t)-\sigma^2(\tilde X(t))\big)\,\dl t.
\end{aligned}
\end{equation}

We estimate the two terms on  the right hand side of \eqref{eq:Euler_error_2} separately. Considering the first term, we have
\begin{equation}\label{eq:Euler_first_order_term}
\begin{aligned}
&\bE\int_0^T\nabla_xF^\epsilon_t(\tilde X_t)\; \big(\tilde b(t,\tilde X_t)-b(\tilde X(t))\big)\,\dl t\\
&=\bE\sum_{n=0}^{N-1}\int_{\tau_n}^{\tau_{n+1}}\nabla_xF^\epsilon_t(\tilde X_t)\; \big(b(\tilde X(\tau_n))-b(\tilde X(t))\big)\,\dl t\\
&=\bE\sum_{n=0}^{N-1}\bE\Big(\int_{\tau_n}^{\tau_{n+1}}\nabla_xF^\epsilon_t(\tilde X_t)\; \big(b(\tilde X(\tau_n))-b(\tilde X(t))\big)\,\dl t\,\Big|\,\cF_{\tau_n}\Big)\\
&=\bE\sum_{n=0}^{N-1}\int_{\tau_n}^{\tau_{n+1}}\Big(\bE\Big[\nabla_xF^\epsilon_t(\tilde X_t^{\tau_n,x_{\tau_n}})\; \big(b(x(\tau_n))-b(\tilde X_t^{\tau_n,x_{\tau_n}})\big)\Big]\Big)\Big|_{x_{\tau_n}=\tilde X_{\tau_n}}\,\dl t.
\end{aligned}
\end{equation}

Let us fix $n\in\{0,\ldots,N-1\}$, $x_{\tau_n}\in C([0,\tau_n];\bR)$ for a while, and let $G=G^{\epsilon,x_{\tau_n}}$ be the non-anticipative functional defined in \eqref{eq:Euler2}. Then, for all $t\in[\tau_n,\tau_{n+1}]$,
\begin{equation}\label{eq:Euler3}
\begin{aligned}
\bE\Big[\nabla_xF^\epsilon_t(\tilde X_t^{\tau_n,x_{\tau_n}})\; \big(b(x(\tau_n))-b(\tilde X_t^{\tau_n,x_{\tau_n}})\big)\Big]=\bE\, G_t(\tilde X^{\tau_n,x(\tau_n)}_t).
\end{aligned}
\end{equation}
Lemma~\ref{lem:Euler1} allows us to expand $G_t(\tilde X^{\tau_n,x(\tau_n)}_t)$ in \eqref{eq:Euler3}  by applying the functional Itô formula: For all $t\in[\tau_n,\tau_{n+1}]$,
\begin{align*}
G_t(\tilde X^{\tau_n,x(\tau_n)}_t)
&=0+\int_{\tau_n}^t\cD_sG(\tilde X^{\tau_n,x(\tau_n)}_s)\,\dl s\\
&\quad+\int_{\tau_n}^t\nabla_x G_s(\tilde X^{\tau_n,x(\tau_n)}_s)\big[b(x(\tau_n))\,\dl s+\sigma(x(\tau_n))\,\dl W(s)\big]\\
&\quad +\frac12\int_{\tau_n}^t\nabla_x^2  G_s(\tilde X^{\tau_n,x(\tau_n)}_s)\,\sigma^2(x(\tau_n))\,\dl s.
\end{align*}
Writing the appearing horizontal and vertical derivatives explicitly according to Lemma~\ref{lem:Euler1} and Lemma~\ref{lem:Euler2} with $n=1$, we obtain
\begin{equation}\label{eq:Euler4}
\begin{aligned}
&G_t(\tilde X^{\tau_n,x(\tau_n)}_t)\\
&=\int_{\tau_n}^t\Big(\nabla_x^2F^\epsilon_s(\tilde X^{\tau_n,x_{\tau_n}}_s)\,b(\tilde X^{\tau_n,x(\tau_n)}(s))
+\frac12\nabla_x^3F^\epsilon_t(\tilde X^{\tau_n,x_{\tau_n}}_s)\,\sigma^2(\tilde X^{\tau_n,x(\tau_n)}(s))\Big)\\
&\qquad\qquad\times\big(b(\tilde X^{\tau_n,x(\tau_n)}(s))-b(x(\tau_n))\big)\,\dl s\\
&\quad+\int_{\tau_n}^t\Big(\nabla_x^2 F_s^\epsilon(\tilde X^{\tau_n,x_{\tau_n}}_s)\big(b(x(\tau_n))-b(\tilde X^{\tau_n,x(\tau_n)}(s))\big)+\nabla_x F_s^\epsilon(\tilde X^{\tau_n,x_{\tau_n}}_s)\,b'(\tilde X^{\tau_n,x(\tau_n)}(s))\Big)\\
&\qquad\qquad\times\big[b(x(\tau_n))\,\dl s+\sigma(x(\tau_n))\,\dl W(s)\big]\\
&\quad +\frac12\int_{\tau_n}^t\Big(\nabla_x^3 F_s^\epsilon(\tilde X^{\tau_n,x_{\tau_n}}_s)\big(b(x(\tau_n))-b(\tilde X^{\tau_n,x(\tau_n)}(s))\big)
+2\nabla_x^2 F_s^\epsilon(\tilde X^{\tau_n,x_{\tau_n}}_s)\,b'(\tilde X^{\tau_n,x(\tau_n)}(s))\\
&\qquad\qquad +\nabla_x F_s^\epsilon(\tilde X^{\tau_n,x_{\tau_n}}_s)\,b^{(2)}(\tilde X^{\tau_n,x(\tau_n)}(s))\Big)\,\sigma^2(x(\tau_n))\,\dl s.
\end{aligned}
\end{equation}
Arguing similarly as in Section~\ref{sec:RegularityFFepsilon}, one can use \eqref{eq:Euler4} to check that there exist constants $C>0$ and $p\geq1$ that do not depend on $n$, $t$, or $\epsilon$ such that
\begin{equation}\label{eq:Euler5}
\begin{aligned}
\big|\bE\,G_t(\tilde X^{\tau_n,x(\tau_n)}_t)\big|&\leq C\int_{\tau_n}^t(1+\|x_{\tau_n}\|_{C([0,\tau_n];\bR)}^p)\,\dl s\\
&\leq C (1+\|x_{\tau_n}\|_{C([0,\tau_n];\bR)}^p)\,(\tau_{n+1}-\tau_n)
\end{aligned}
\end{equation}
for all $x_{\tau_n}\in C([0,\tau_n];\bR)$.
Plugging \eqref{eq:Euler5} and \eqref{eq:Euler3} into \eqref{eq:Euler_first_order_term} and using the fact that $\|\tilde X_T\|_{C([0,T];\bR)}$ has finite moments of all orders (as Burkholder's inequality and an application of Gronwall's lemma show) we finally obtain the estimate
\begin{equation}\label{eq:Euler_6}
\begin{aligned}
\Big|\bE\int_0^T\nabla_xF^\epsilon_t(\tilde X_t)\; \big(\tilde b(t,\tilde X_t)-b(\tilde X(t))\big)\,\dl t\Big|
\leq C\,\delta
\end{aligned}
\end{equation}
with a constant $C$ that does not depend on $\epsilon$ or $\delta$.

The second term on the right hand side of \eqref{eq:Euler_error_2} can be treated in complete analogy to the first term, this time using Lemma \ref{lem:Euler1a} and Lemma \ref{lem:Euler2} with $n=2$, yielding the estimate
\begin{equation}\label{eq:Euler_7}
\begin{aligned}
\Big|\frac12\bE\int_0^T\nabla_x^2 F^\epsilon_t(\tilde X_t)\big(\tilde \sigma^2(t,\tilde X_t)-\sigma^2(\tilde X(t))\big)\,\dl t\Big|
\leq C\,\delta
\end{aligned}
\end{equation}
with a constant $C$ that does not depend on $\epsilon$ or $\delta$. As no new arguments are needed, we omit the details of the proof of \eqref{eq:Euler_7}.

Finally, the combination of \eqref{eq:Euler_8}, \eqref{eq:Euler_error_2}, \eqref{eq:Euler_6} and \eqref{eq:Euler_7}, as the constant $C$ in  \eqref{eq:Euler_6} and \eqref{eq:Euler_7} is independent of $\epsilon$, finishes the proof.
\end{proof}

\begin{appendix}
\section{Appendix}

\begin{lemma}\label{lem:LpConvImpliesWeakConv}
Let $(B,\nnrm{\cdot}B)$ be a real Banach space, $(S,\nnrm{\cdot}S)$ a normed real vector space, and $\varphi\in C_p(B,S)$. Let $Y,\,Y_n\in L^p(\Omega;B)$, $n\in\bN$, such that $Y_n\xrightarrow{n\to\infty} Y$ in $L^p(\Omega;B)$ for all $p\geq1$. Then, for all $p\geq1$,
\[\bE(\nnrm{\varphi(Y_n)-\varphi(Y)}S^p)\xrightarrow{n\to\infty}0.\]
\end{lemma}

\begin{proof}
For $R\in(0,\infty)$ let $\eta_R\in C(B,\bR)$ be a cut-off function such that $\eta_R(x)=1$ for $\nnrm{x}B\leq R$, $\eta_R(x)=0$ for $\nnrm{x}B\geq R+1$, and  $\eta_R(B)=[0,1]$. Define $\varphi_R,\,\varphi^R\in C(B,S)$ by
\[\varphi_R:=\eta_R\,\varphi,\quad \varphi^R:=(1-\eta_R)\,\varphi.\]
We have
\begin{equation*}
\bE(\nnrm{\varphi(Y_n)-\varphi(Y)}S^p)
\leq 2^{p-1}\Big(\bE(\nnrm{\varphi_R(Y_n)-\varphi_R(Y)}S^p)+\bE(\nnrm{\varphi^R(Y_n)-\varphi^R(Y)}S^p)\Big).
\end{equation*}

To handle the term $\bE(\nnrm{\varphi_R(Y_n)-\varphi_R(Y)}S^p)$, we use $\psi\in C_b(B\times B,\bR)$ defined by
\[\psi(x,y):=\nnrm{\varphi_R(x)-\varphi_R(y)}S^p,\quad x,\,y\in B.\]
On the product space $B\times B$, we consider the product topology and the norm
$\nnrm{(x,y)}{B\times B}:=\nnrm{x}B+\nnrm{y}B$.
The convergence $\bE(\nnrm{Y_n-Y}B)\xrightarrow{n\to\infty}0$ implies that $\nnrm{(Y_n,Y)-(Y,Y)}{B\times B}=\nnrm{Y_n-Y}B\xrightarrow{n\to\infty}0$ in probability.
It follows that $\bP_{(Y_n,Y)}\xrightarrow{n\to\infty}\bP_{(Y,Y)}$ weakly, and in particular
\[\bE(\nnrm{\varphi_R(Y_n)-\varphi_R(Y)}S^p)=\bE\, \psi(Y_n,Y)\xrightarrow{n\to\infty}\bE\,\psi(Y,Y)=0.\]

To finish the proof it suffices to show that $\sup_{n\in\bN}\bE(\nnrm{\varphi^R(Y_n)-\varphi^R(Y)}S^p)$ tends to zero as $R\to\infty$. The polynomial growth of $\varphi:B\mapsto S$ implies that there exist $C,q\in[1,\infty)$ such that
\begin{align*}
\sup_{n\in\bN}\,&\bE(\nnrm{\varphi^R(Y_n)-\varphi^R(Y)}S^p)\\
&\leq C\sup_{n\in\bN}\Big(\int_{\{\nnrm{Y_n}B\geq R\}}(1+\nnrm{Y_n}B^q)\,\dl\bP+\int_{\{\nnrm{Y}B\geq R\}}(1+\nnrm{Y}B^q)\,\dl\bP\Big)\\
&\leq C \sup_{n\in\bN_0}\int_{\{\nnrm{Y_n}B\geq R\}}(1+\nnrm{Y_n}B^q)\,\dl\bP,
\end{align*}
where we have set $Y_0:=Y$.
The last term tends to zero as $R\to\infty$ since $(1+\nnrm{Y_n}B^q)_{n\in\bN_0}$ is bounded in $L^r(\Omega;\bR)$ for every $r\in[1,\infty)$ and hence uniformly integrable.
\end{proof}

\begin{lemma}\label{lem:topsup}
Under Assumption~\ref{ass:bsigma}, the topological support of $\bP_{X_T}$ in $C([0,T],\bR^d)$ is $\{x\in C([0,T],\bR^d):~x(0)=\xi_0\}$.
\end{lemma}
\begin{proof}
The statement is a straightforward consequence of a general version the Stroock-Varadhan support theorem \cite[Theorem 3.1]{Gyongy90} (for the original theorem see \cite{SW72}, see also \cite{MSS94}).
Let $H$ be the space of the absolutely continuous functions $\omega\colon[0,T]\to \bR^m$ with 
$\omega(0)=0$. For $\omega\in H$, consider the ordinary differential equation
\begin{equation}\label{eq:ode}
\begin{aligned}
\dot{x}^{\omega}(t)&=b(x^{\omega}(t))-\frac{1}{2}(\nabla \sigma)\sigma(x^{\omega}(t))+\sigma(x^\omega(t))\dot{\omega}(t)\\
x^{\omega}(0)&=\xi_0,
\end{aligned}
\end{equation}
Here the $i$-th coordinate of the vector $(\nabla \sigma)\sigma(x)\in \mathbb{R}^d$ is given by
$$
[(\nabla \sigma)\sigma(x)]_i=\sum_{k=1}^d\sum_{j=1}^m\big(\frac{\partial}{\partial x_k}\sigma_{i,j}(x)\big)\sigma_{k,j}(x).
$$ 
By \cite[Theorem 3.1]{Gyongy90}, under our assumptions on $b$ and $\sigma$, the topological support of $\bP_X$ in $(C([0,T],\bR^d),\nnrm{\cdot}\infty)$ is
the closure of the set $\{x^{\omega}\in C([0,T],\bR^d): ~\omega \in H\}$ (the factor $\frac{1}{2}$ is missing from \eqref{eq:ode} in \cite{Gyongy90} due to a typo). 
Let $x$ be an absolutely continuous function from $[0,T]$ to $\mathbb{R}^d$ with 
$x(0)=\xi_0$ and set
$a(x(s)):=\sigma(x(s))^{\top}[\sigma(x(s))\sigma(x(s))^{\top}]^{-1}$. Define
$$
\omega(t)=\int_{0}^t\big(a(x(s))\dot{x}(s)-a(x(s))b(x(s))+\frac{1}{2}a(x(s))(\nabla \sigma)\sigma(x(s))\big)\,\dl s.
$$
Then $\omega\in H$, and
$$
\dot{\omega}(t)=a(x(t))\dot{x}(t)-a(x(t))b(x(t))+\frac{1}{2}a(x(t))(\nabla \sigma)\sigma(x(t))
$$
whence
$$
\dot{x}(t)=b(x(t))-\frac{1}{2}(\nabla \sigma)\sigma(x(t))+\sigma(x(t))\dot{\omega}(t).
$$
Therefore, $$\{x\text{ is abs. continuous from } [0,T]\text{ to }\mathbb{R}^d: x(0)=\xi_0\}\subset \{x^{\omega}\in C([0,T],\bR^d): ~\omega \in H\}$$ and the statement follows by taking closures in $(C([0,T],\bR^d),\nnrm{\cdot}\infty)$.
\end{proof}

\end{appendix}

\providecommand{\bysame}{\leavevmode\hbox to3em{\hrulefill}\thinspace}

\section*{}
\noindent Mih\'{a}ly Kov\'{a}cs\\
Department of  Mathematics and Statistics\\
University of Otago \\
P.O.~Box 56, Dunedin, New Zealand\\
E-mail: mkovacs@maths.otago.ac.nz \\

\noindent Felix Lindner\\
Fachbereich Mathematik \\
Technische Universit\"{a}t Kaiserslautern\\
Postfach 3049, 67653 Kaiserslautern, Germany \\
E-mail: lindner@mathematik.uni-kl.de
\end{document}